\documentclass[11pt]{article}

\usepackage{amsmath,amssymb,amsthm,bm,cite,color,paralist,graphicx,subcaption,float}

\theoremstyle{plain}
\newtheorem{thm}{Theorem}
\newtheorem{lem}[thm]{Lemma}

\newtheorem{prob}[thm]{Problem}

\def\BA{\mathbb A}
\def\BB{\mathbb B}
\def\BC{\mathbb C}
\def\BN{\mathbb N}
\def\BR{\mathbb R}
\def\cA{\mathcal A}

\def\T{\mathrm T}

\def\rd{\mathrm d}

\def\rdiv{\mathrm{div}}

\def\op{\mathrm{op}}

\def\Ga{\Gamma}

\def\Om{\Omega}
\def\al{\alpha}
\def\be{\beta}
\def\ga{\gamma}

\def\ka{\kappa}
\def\la{\lambda}

\def\vp{\varphi}

\def\f{\frac}
\def\nb{\nabla}
\def\ov{\overline}
\def\pa{\partial}

\allowdisplaybreaks[4]

\title{Forward and Backward Problems for Coupled Subdiffusion Systems}

\author{Dian Feng\thanks{School of Mathematical Sciences, Fudan University, 220 Handan Road, Shanghai 200433, China. E-mail: {\tt dfeng19@fudan.edu.cn}}\quad
Yikan Liu\thanks{Corresponding author. Department of Mathematics, Kyoto University,
Kitashirakawa-Oiwakecho, Sakyo-ku, Kyoto 606-8502, Japan.
E-mail: {\tt liu.yikan.8z@kyoto-u.ac.jp}}\quad
Shuai Lu\thanks{School of Mathematical Sciences, Fudan University, 220 Handan Road, Shanghai 200433, China. E-mail: {\tt slu@fudan.edu.cn}}}

\date{\footnotesize Dedicated to Bernd Hofmann on the occasion of his 70th birthday.}

\begin{document}

\maketitle

\baselineskip 18pt

\begin{abstract}
In this article, we investigate both forward and backward problems for coupled systems of time-fractional diffusion equations, encompassing scenarios of strong coupling. For the forward problem, we establish the well-posedness of the system, leveraging the eigensystem of the corresponding elliptic system as the foundation. When considering the backward problem, specifically the determination of initial values through final time observations, we demonstrate a Lipschitz stability estimate, which is consistent with the stability observed in the case of a single equation. To numerically address this backward problem, we refer to the explicit formulation of Tikhonov regularization to devise a multi-channel neural network architecture. This innovative architecture offers a versatile approach, exhibiting its efficacy in multidimensional settings through numerical examples and its robustness in handling initial values that have not been trained.
\medskip

\medskip

\end{abstract}

%%%%%%%%%%%%%%%%%%%%%%%%%%%%%%%%%%%%%%%%
\section{Introduction}

The last decade has witnessed explosive developments of nonlocal models from various practical backgrounds. As a popular representative, partial differential equations with fractional derivatives in time have gathered considerable popularity in modeling a diverse array of significant diffusion phenomena, represented by prominent examples such as subdiffusion in underground environment with high heterogeneity \cite{HH} and relaxation phenomena in complex viscoelastic materials \cite{GCR}. Recently, linear theory for fractional diffusion equations with fractional time derivatives ranging in $(0,1)$ has been well established, followed by extensive researches on related numerical methods and inverse problems (see e.g.\! \cite{SY11,KRY20,JZ23,KR23}, just list a few). Nevertheless, there seems much less literature investigating coupled systems of subdiffusion equations. This motivates the generalizations of the studies on time-fractional diffusion equations from single equations to coupled systems.

In our current work, we propose and examine forward and backward problems for coupled systems of time-fractional diffusion equations. To formulate the problem, we first define the time-fractional differential operator. Let $\al\in(0,1)$, $T>0$ be constants and $\rd_t^\al:L^2(0,T)\longrightarrow L^2(0,T)$ be the primitive $\al$-th order Caputo differential operator defined by (e.g., Podlubny \cite{P99})
\[
D(\rd_t^\al)=\{f\in C^1[0,T]\mid f(0)=0\},\quad\rd_t^\al f(t):=\int_0^t\f{(t-s)^{-\al}}{\Ga(1-\al)}\f{\rd f}{\rd s}(s)\,\rd s,
\]
where $D(\,\cdot\,)$ and $\Ga(\,\cdot\,)$ denote the domain of an operator and the Gamma function, respectively. By $\pa_t^\al:L^2(0,T)\longrightarrow L^2(0,T)$ we denote the smallest closed extension of $\rd_t^\al$. Then it follows from \cite[Theorem 2.5]{KRY20} that $\pa_t^\al$ is a natural generalization of $\rd_t^\al$ in the Sobolev-Slobodeckij space
\[
H^\al(0,T):=\left\{f\in L^2(0,T)\mid\|f\|_{H^\al(0,T)}:=\left(\int_0^T\!\!\!\int_0^T\f{|f(t)-f(s)|^2}{|t-s|^{1+2\al}}\,\rd t\rd s\right)^{1/2}<\infty\right\}.
\]
More precisely, we know
\begin{equation}\label{eq-def-Hal}
D(\pa_t^\al)=H_\al(0,T):=\left\{\!\begin{alignedat}{2}
& H^\al(0,T), & \quad & 0<\al<1/2,\\
& \left\{f\in H^{1/2}(0,T)\mid\int_0^T\f{|f(t)|^2}t\,\rd t<\infty\right\}, & \quad & \al=1/2,\\
& \{f\in H^\al(0,T)\mid f(0)=0\}, & \quad & 1/2<\al\le1
\end{alignedat}\right.
\end{equation}
and $\pa_t^\al$ is an isomorphism between $H_\al(0,T)$ and $L^2(0,T)$. Furthermore, $H_\al(0,T)$ is a Banach space equipped with the norm
\[
\|f\|_{H_\al(0,T)}:=\left\{\!\begin{alignedat}{2}
& \|f\|_{H^\al(0,T)}, & \quad & \al\ne1/2,\\
& \left(\|f\|_{H^{1/2}(0,T)}^2+\int_0^T\f{|f(t)|^2}t\,\rd t\right)^{1/2}, & \quad & \al=1/2.
\end{alignedat}\right.
\]

With the above defined fractional differential operator $\pa_t^\al$, we formulate the initial-boundary value problems under consideration. Let $\Om\subset\BR^d$ ($d\in\BN:=\{1,2,\dots\}$) be a bounded domain with a smooth boundary $\pa\Om$. In this article, we investigate both
\begin{equation}\label{eq-IBVP1}
\left\{\!\begin{alignedat}{2}
& \pa_t^\al\left(u_k-u_0^{(k)}\right)-\rdiv(\bm A_k(\bm x)\nb u_k)+\sum_{\ell=1}^K c_{k\ell}(\bm x)u_\ell=F_k & \quad & \mbox{in }\Om\times(0,T),\\
& u_k=0 & \quad & \mbox{on }\pa\Om\times(0,T)
\end{alignedat}\right.
\end{equation}
with $k=1,\dots,K$ for a constant $K\in\BN$ and
\begin{equation}\label{eq-IBVP2}
\begin{cases}
\pa_t^\al(\bm u-\bm u_0)-\rdiv(\BA(\bm x)\nb\bm u)+\bm C(\bm x)\bm u=\bm F & \mbox{in }\Om\times(0,T),\\
\bm u=\bm0 & \mbox{on }\pa\Om\times(0,T).
\end{cases}
\end{equation}
In \eqref{eq-IBVP1}, it is assumed that each $\bm A_k$ is a strictly positive definite matrix-valued function on $\ov\Om$ of $C^1$ class and $\bm C:=(c_{k\ell})_{1\le k,\ell\le K}$ is a non-negative definite matrix-valued function in $\Om$ of $L^\infty$ class. More precisely, there exists a constant $\ka>0$ such that
\begin{equation}\label{eq-cond-Ak}
 \bm A_k\in C^1(\ov\Om;\BR^{d\times d}),\quad\bm A_k(\bm x)=(\bm A_k(\bm x))^\T,\quad\bm A_k\bm\xi\cdot\bm\xi\ge\ka|\bm\xi|^2
\end{equation}
for any $\bm\xi\in\BR^d$, $\bm x\in\ov\Om$, $k=1,\dots,K$ and
\begin{equation}\label{eq-cond-C}
\bm C=(c_{k\ell})_{1\le k,\ell\le K}\in L^\infty(\Om;\BR^{K\times K}),\quad\bm C(\bm x)=(\bm C(\bm x))^\T,\quad\bm C\bm\xi\cdot\bm\xi\ge0
\end{equation}
for any $\bm\xi\in\BR^d$ and a.e.\! $\bm x\in\Om$, where $(\,\cdot\,)^\T$ and $|\cdot|^2$ denote the transpose and the Euclidean distance, respectively. On the other hand, $u_0^{(k)}$ stands for the initial value of $u_k$ ($k=1,\dots,K$) in the sense that $u_k(\bm x,\,\cdot\,)-u_0^{(k)}(\bm x)$ lies in the domain $H_\al(0,T)$ of $\pa_t^\al$ for a.e.\! $\bm x\in\Om$, which is easily understood in view of \eqref{eq-def-Hal} with $\al>1/2$. In the sequel, we always abbreviate
\[
\bm u:=(u_1,\dots,u_K)^\T,\quad\bm u_0:=(u_0^{(1)},\dots,u_0^{(K)})^\T,\quad\bm F:=(F_1,\dots,F_K)^\T.
\]

In \eqref{eq-IBVP2}, we similarly denote $\bm u_0=(u_0^{(1)},\dots,u_0^{(d)})^\T$ as the initial value of $\bm u=(u_1,\dots,u_d)^\T$ as explained above. Meanwhile, we assume that $\bm C$ satisfies \eqref{eq-cond-C} with $K=d$, and $\BA:=(a_{ijk\ell})_{1\le i,j,k,\ell\le d}\in C^1(\ov\Om;\BR^{d\times d\times d\times d})$ is a fully symmetric fourth order tensor on $\ov\Om$ satisfying the stability condition (see \cite{LQ12}), that is,
\begin{equation}\label{eq-fullsym}
a_{ijk\ell}=a_{jik\ell}=a_{k\ell ij},\quad\forall\,i,j,k,\ell=1,\dots,d\mbox{ on }\ov\Om
\end{equation}
and there exists a constant $\ka>0$ such that
\begin{equation}\label{eq-stabcond}
 \sum_{i,j,k\ell=1}^d a_{ijk\ell}(\bm x)e_{ij}e_{k\ell}\ge\ka\sum_{i,j=1}^d e_{ij}^2
\end{equation}
for any $\bm x\in\ov\Om$ and symmetric matrix $(e_{ij})_{1\le i,j\le d}\in\BR^{d\times d}$. Both $\nb\bm u$ and $\BA(\bm x)\nb\bm u$ are matrix-valued functions of $d\times d$, whose entries are
\[
(\nb\bm u)_{k\ell}=\pa_\ell u_k\ (1\le k,\ell\le d),\quad (\BA(\bm x)\nb\bm u)_{ij}=\sum_{k,\ell=1}^d a_{ijk\ell}(\bm x)\pa_\ell u_k\ (1\le i,j\le d),
\]
respectively. Finally, $\rdiv(\BA(\bm x)\nb\bm u)$ is a column vector with
\[
(\rdiv(\BA(\bm x)\nb\bm u))_i=\sum_{j=1}^d\pa_j(\BA(\bm x)\nb\bm u)_{ij}\quad(1\le i\le d).
\]

Both \eqref{eq-IBVP1} and \eqref{eq-IBVP2} require the coincidence of the orders of time derivatives in all components, in which \eqref{eq-IBVP1} is a weakly coupled system. In view of the generality, the formulation \eqref{eq-IBVP1} is indeed a special case of that discussed in \cite{LHL23}, where the orders of time derivatives are allowed to be different among components. In contrast, \eqref{eq-IBVP2} restricts $K=d$ but allows strong coupling up to the second order derivatives in space. Especially, the second order term $-\rdiv(\BA(\bm x)\nb\bm u)$ appears as a standard formulation in hyperbolic systems modeling the elasticity for $d=2,3$. Therefore, the formulation \eqref{eq-IBVP2} is not only parallel to that in \cite{LHL23}, but also serves as a preliminary research of time-fractional wave equation  with $1<\al<2$ modeling the viscoelastic materials (e.g. \cite{BDES18, KR22}).

In this article, our first objective is to establish the well-posedness of the problems defined in equations \eqref{eq-IBVP1} and \eqref{eq-IBVP2}. As the second subject of the current work, we investigate the following inverse problem.

\begin{prob}[Backward problem]\label{prob1}
Let $T>0$ be arbitrary finite time and $\bm u$ satisfy \eqref{eq-IBVP1} or \eqref{eq-IBVP2} with $\bm F\equiv\bm0$. Determine the initial value $\bm u_0$ by the final observation of $\bm u$ at $t=T$.
\end{prob}

In Problem \ref{prob1}, we attempt to determine the initial values of all components in \eqref{eq-IBVP1} or \eqref{eq-IBVP2} simultaneously by observing the state of all components at the same final moment $T$. As possible generalizations, it seems interesting to study the same problem with the observation taken at different moments $T_k$ for the component $u_k$. Moreover, owing to the coupling effect between components, one can also study the possibility of solving Problem \ref{prob1} with a part or even one of the observed components. However, such issues require further investigations and for the moment we restrict ourselves to the framework of Problem \ref{prob1}.

It is  worth mentioning that as one of the most typical inverse problems, there has been considerable research on backward problems for single subdiffusion equation, initiated from the pioneering works of Liu and Yamamoto \cite{LY10}. We further refer to the papers \cite{RXL14,FLY20,FY20} as well as a comprehensive survey \cite[\S5.2]{LY19} especially highlighting the numerical studies prior to the year of 2019. Recently, \cite{LY23} obtained the uniqueness of the backward problem for a multi-term subdiffusion equation by the short-time behavior of the solution. Meanwhile, \cite{CMY24} established a conditional H\"older stability for a backward problem associated with the fractional evolution equation $\pa_t^\al u+A u=0$ with a general self-adjoint operator $A$. On the other hand, inverse problems for coupled systems of subdiffusion equations have also started to gain certain attention, including the studies in \cite{RHY21,LHL23}. However, to the authors' best knowledge, there seems no existing investigation into the backward problem specifically for the strongly coupled case. This area remains unexplored and provides opportunities for further research.

In the numerical aspect, although we can solve \eqref{eq-IBVP1},  \eqref{eq-IBVP2} and their corresponding backward problems using classical numerical schemes and regularization methods, these traditional approaches are limited to solving specific instances of the equations. In this article, we aim at developing a generic neural network architecture that learns the fundamental characteristics of backward problems for subdiffusion process.

Over the past few years, a wide range of neural network architectures have been developed to solve partial differential equations and their inverse problems, including the deep Ritz method \cite{B18}, the deep Galerkin method \cite{SS18}, physics-informed neural networks (PINNs) \cite{RP19}, weak adversarial networks \cite{ZB20} and others \cite{KL21,LL20,LL18}. However, applying automatic differentiation techniques from machine learning to fractional differential equations is more complicated than applying them to conventional derivatives. Indeed, fractional orders are inherently defined via integral representations, requiring additional methods within neural networks to accurately handle fractional derivatives. These approaches often involve techniques such as Monte Carlo sampling or finite difference approximations. While several studies have explored this field, for instance, the work presented in \cite{PL19,GW22}, it is noteworthy that most existing results focused on the single equation, with limited exploration of coupled systems. The neural network approach introduced in this paper does not attempt to directly solve the backward problem using traditional numerical methods. Instead, its objective is to learn the complex interdependencies and correlations between diffusion components at different time within coupled systems.

From the theoretical aspect, there is no essential difference between our formulation of coupled systems with single equations in literature due to the self-adjoint structure on the spatial direction, so that classical methods such as eigenfunction expansions still work. Nevertheless, such verification needs some advanced tools, like Korn's inequality or m-accretive operators as demonstrated in Appendix \ref{sec-app}. Meanwhile, the main novelty of this article lies in the application of neural network to backward problems, which has never been achieved for fractional equations before. Moreover, we implemented the reconstruction for coupled systems with variable diffusion coefficients, while almost all literature only involved a single Laplacian for single equations.

This article will be structured as follows. In Section \ref{sec-main}, we will outline the main theoretical contributions and their respective proofs. Specifically, we will establish the well-posedness for the solution of coupled systems and derive the stability estimate for the backward problem (Problem \ref{prob1}). Section \ref{sec-numer} will focus on the numerical investigation of the backward problem.
We will develop a neural network inversion algorithm based on the concept of conventional Tikhonov regularization. By generating a training data set using classical numerical methods, we will evaluate the accuracy and generalization capabilities of this algorithm. Finally, we close this article with an appendix providing a detailed discussion on the existence of an eigensystem of the elliptic part $\cA$ in our formulation along with its fractional powers.

%%%%%%%%%%%%%%%%%%%%%%%%%%%%%%%%%%%%%%%%
\section{Main Results}\label{sec-main}

In this section, we provide the statements of main results in this article along with their proofs. To this end, we start with some preparations. First we recall the frequently used Mittag-Leffler function
\[
E_{\al,\be}(z):=\sum_{m=0}^\infty\f{z^m}{\Ga(\al m+\be)},\quad\be\in\BR,\ z\in\BC.
\]

Next, let $(L^2(\Om))^K$ be the product space of $L^2(\Om)$, which is a Hilbert space with the inner product
\[
(\bm f,\bm g):=\int_\Om\bm f\cdot\bm g\,\rd\bm x=\sum_{k=1}^K\int_\Om f_k g_k\,\rd\bm x
\]
for $\bm f=(f_1,\dots,f_K)^\T,\bm g=(g_1,\dots,g_K)^\T\in(L^2(\Om))^K$. Then the norm of $(L^2(\Om))^K$ is induced as $\|\bm f\|_{(L^2(\Om))^K}:=(\bm f,\bm f)^{1/2}$, which is abbreviated as $\|\bm f\|_{L^2(\Om)}$ if there is no fear of confusion. It turns out that initial-boundary value problems \eqref{eq-IBVP1} and \eqref{eq-IBVP2}, though complicated, admit a unified formulation. In fact, for \eqref{eq-IBVP1} we introduce
\begin{gather*}
D(\cA_k):=H^2(\Om)\cap H_0^1(\Om),\quad\cA_k:D(\cA_k)\longrightarrow L^2(\Om),\\
\cA_k f:=-\rdiv(\bm A_k(\bm x)\nb f)\quad(k=1,\dots,K),\\
D(\cA):=(H^2(\Om)\cap H_0^1(\Om))^K,\quad\cA:=\mathrm{diag}(\cA_1,\dots,\cA_K)+\bm C:D(\cA)\longrightarrow(L^2(\Om))^K.
\end{gather*}
Then it is readily seen that \eqref{eq-IBVP1} can be rewritten as
\begin{equation}\label{eq-IBVP0}
\begin{cases}
\pa_t^\al(\bm u-\bm u_0)+\cA\bm u=\bm F & \mbox{in }\Om\times(0,T),\\
\bm u=\bm 0 & \mbox{on }\pa\Om\times(0,T).
\end{cases}
\end{equation}
For \eqref{eq-IBVP2}, we immediately arrive at \eqref{eq-IBVP0} by simply putting
\[
\cA\bm u:=-\rdiv(\BA(\bm x)\nb\bm u)+\bm C(\bm x)\bm u.
\]
The common representation \eqref{eq-IBVP0} not only provides convenience for discussing \eqref{eq-IBVP1} and \eqref{eq-IBVP2} simultaneously, but also enables a similar treatment for coupled systems as that for a single equation.

Actually, in both cases of \eqref{eq-IBVP1} and \eqref{eq-IBVP2}, there exists an eigensystem $\{(\la_n,\bm\vp_n)\}_{n=1}^\infty$ of $\cA$ such that
\[
\begin{cases}
\cA\bm\vp_n=\la_n\bm\vp_n & \mbox{in }\Om,\\
\bm\vp_n=\bm0 & \mbox{on }\pa\Om
\end{cases}\ (n=1,2,\dots),\quad0<\la_1\le\la_2\le\cdots,\ \la_n\longrightarrow\infty\ (n\to\infty),
\]
and $\{\bm\vp_n\}\subset D(\cA)$ forms a complete orthonormal system of $(L^2(\Om))^K$. Furthermore, one can define the fractional power $\cA^\ga$ of $\cA$ along with its domain $D(\cA^\ga)$ for $\ga\in(0,1)$ by
\begin{gather}
D(\cA^\ga):=\left\{\bm f\in(L^2(\Om))^K\mid\|\bm f\|_{D(\cA^\ga)}:=\|\cA^\ga\bm f\|_{L^2(\Om)}<\infty\right\},\label{eq-domain-fp}\\
\cA^\ga\bm f:=\sum_{n=1}^\infty\la_n^\ga(\bm f,\bm\vp_n)\bm\vp_n.\label{eq-fracpow}
\end{gather}
We know that $D(\cA^\ga)$ is a Hilbert space and especially $D(\cA^{1/2})=(H_0^1(\Om))^K$. By $H^{-1}(\Om)$ we denote the dual of $H_0^1(\Om)$ in the topology of $L^2(\Om)$.

The existence of the eigensystem $\{(\la_n,\bm\vp_n)\}$ and fractional powers of $\cA$ follows from standard theories for eigenvalue problems and fractional powers of m-accretive operators, which can be found e.g.\! in \cite[\S6.5]{E10}, \cite[\S2.6]{P83} and \cite[\S3]{T79}. Nevertheless, for the sake of self-containedness, we provide detailed verification in Appendix \ref{sec-app}.

As long as the eigensystem of $\cA$ is available, we can follow the same line of \cite{SY11} to establish the basic well-posedness results for \eqref{eq-IBVP0}.

\begin{thm}\label{thm-FP}
Let $\bm u_0\in(L^2(\Om))^K,$ $\bm F\in L^2(0,T;(L^2(\Om))^K)$ and choose any $\ga\in[0,1]$.
\begin{enumerate}
\item If $\bm F\equiv\bm0,$ then the initial-boundary value problem \eqref{eq-IBVP0} admits a unique solution $\bm u\in L^{1/\ga}(0,T;D(\cA^\ga)),$ where $1/\ga=\infty$ for $\ga=0$. Furthermore, $\bm u$ satisfies
\[
\lim_{t\to0+}\|\bm u-\bm u_0\|_{L^2(\Om)}=0,\quad\bm u-\bm u_0\in H_\al(0,T;(H^{-1}(\Om))^K)
\]
and allows the representation
\begin{equation}\label{eq-solrep1}
\bm u(\,\cdot\,,t)=\sum_{n=1}^\infty E_{\al,1}(-\la_n t^\al)(\bm u_0,\bm\vp_n)\bm\vp_n\quad\mbox{in }\bigcap_{0\le\ga\le1}L^{1/\ga}(0,T;D(\cA^\ga)).
\end{equation}
Moreover, there exists a constant $C>0$ depending only on $\Om,\al,\cA$ such that
\begin{gather}
\|\bm u(\,\cdot\,,t)\|_{D(\cA^\ga)}\le C\|\bm u_0\|_{L^2(\Om)}t^{-\al\ga},\quad0<t<T,\label{eq-solest1}\\
\|\bm u\|_{L^{1/\ga}(0,T;D(\cA^\ga))}\le C\left(\f{T^{1-\al}}{1-\al}\right)^\ga\|\bm u_0\|_{L^2(\Om)},\label{eq-solest2}\\
\|\bm u-\bm u_0\|_{H_\al(0,T;(H^{-1}(\Om))^K)}\le C\left(\f{T^{1-\al}}{1-\al}\right)^{1/2}\|\bm u_0\|_{L^2(\Om)}.\nonumber%label{eq-solest3}
\end{gather}
\item If $\bm u_0\equiv\bm0,$ then the initial-boundary value problem \eqref{eq-IBVP0} admits a unique solution $\bm u\in L^2(0,T;D(\cA))\cap H_\al(0,T;(L^2(\Om))^K),$ which allows the representation
\begin{equation}\label{eq-solrep2}
\bm u(\,\cdot\,,t)=\sum_{n=1}^\infty\left(\int_0^t s^{\al-1}E_{\al,\al}(-\la_n s^\al)(\bm F(\,\cdot\,,t-s),\bm\vp_n)\,\rd s\right)\bm\vp_n
\end{equation}
in $L^2(0,T;D(\cA))\cap H_\al(0,T;(L^2(\Om))^K)$. Moreover, there exists a constant $C_T$ depending only on $\Om,\al,\cA$ and $T$ such that
\begin{equation}\label{eq-solest4}
\|\bm u\|_{L^2(0,T;D(\cA))}+\|\bm u\|_{H_\al(0,T;(L^2(\Om))^K)}\le C_T\|\bm F\|_{L^2(0,T;(L^2(\Om))^K)}.
\end{equation}
\end{enumerate}
\end{thm}

Using the solution representation \eqref{eq-solrep1}, we can immediately obtain the Lipschitz stability for Problem \ref{prob1} below.

\begin{thm}\label{thm-BP}
Let $\bm F\equiv\bm0$ and
choose a finite $T>0$, $\bm u_1\in D(\cA)$ arbitrarily. Then there exists a unique $\bm u_0\in(L^2(\Om))^K$ such that the solution $\bm u$ to \eqref{eq-IBVP0} with the initial value $\bm u_0$ satisfies $\bm u(\,\cdot\,,T)=\bm u_1$ in $\Om$. Moreover, there exists a constant $C_T'>0$ depending only on $\Om,\al,\cA$ and $T$ such that
\[
\|\bm u_0\|_{L^2(\Om)}\le C_T'\|\bm u_1\|_{D(\cA)}.
\]
\end{thm}

As one can observe in the theorem presented above, the stability for the backward problem (Problem 1) exhibits a Lipschitz-type behavior, which aligns with the stability characteristics of the backward problem for a single diffusion equation, as discussed in \cite{LY10, FLY20, FY20}. This Lipschitz stability estimate contrasts with the logarithmic stability encountered in parabolic equations, as noted in \cite{I06}, and it provides insights into the memory effect inherent in the subdiffusion process.

Below we provide the proofs of Theorems \ref{thm-FP}--\ref{thm-BP}. Although the arguments resemble those for single equations (e.g., \cite{SY11}) owing to the self-adjoint structure on the spatial direction, we still sketch the proofs for the sake of completeness. To this end, we summarize some frequently used properties of the Mittag-Leffler functions.

\begin{lem}[\cite{P99,SY11}]\label{lem-ML}
Let $\al\in(0,1),$ $\be\ge\al$ and $\la>0$ be constants.
\begin{enumerate}
\item There exists a constant $C=C(\al,\be)>0$ such that
\[
|E_{\al,\be}(-\eta)|\le\f C{1+\eta},\quad\forall\,\eta\ge0.
\]
\item There holds
\[
\f\rd{\rd t}E_{\al,1}(-\la t^\al)=-\la t^{\al-1}E_{\al,\al}(-\la t^\al),\quad t>0.
\]
\item There holds $E_{\al,1}(-\eta)>0$ for any $\eta\ge0$.
\item If $\mu$ is any real number satisfying $\frac{\pi\alpha}{2}<\mu<\min\{\pi,\pi\alpha\}$, and $p\geq 1$ is an arbitrary integer, then as $|z|\to \infty$ and $\mu\leq |arg(z)|\leq \pi$, the following relation holds
\[
E_{\alpha,\beta}(z)=-\sum_{k=1}^p \frac{z^{-k}}{\Gamma(\beta-\alpha k)} + O\left( |z|^{-1-p} \right).
\]
\end{enumerate}
\end{lem}

\begin{proof}[Proof of Theorem $\ref{thm-FP}$]
(i) Following the same line as that in \cite{SY11}, one can readily verify that \eqref{eq-solrep1} actually provides the unique solution to \eqref{eq-IBVP0} when $\bm F=\bm0$. Using the expression for $\bm u$ in \eqref{eq-solrep1} and the definition of $D(\cA^\ga)$, we employ Lemma \ref{lem-ML}(i) to estimate
\begin{align*}
\|\bm u(\,\cdot\,,t)\|^2_{D(\cA^\ga)} & =\left\|\sum_{n=1}^\infty E_{\al,1}(-\la_n t^\al)(\bm u_0,\bm\vp_n)\la_n^\ga\bm\vp_n\right\|_{L^2(\Om)}^2=\sum_{n=1}^\infty\left|E_{\al,1}(-\la_n t^\al)\la_n^\ga(\bm u_0,\bm\vp_n)\right|^2\\
& \le\sum_{n=1}^\infty\left(\f C{1+\la_n t^\al}\right)^2\left|\la_n^\ga(\bm u_0,\bm\vp_n)\right|^2=C^2\sum_{n=1}^\infty\left(\f{(\la_n t^\al)^\ga}{1+\la_n t^\al}\right)^2\left|(\bm u_0,\bm\vp_n)t^{-\al\ga}\right|^2\\
& \le(C\,t^{-\al\ga})^2\sum_{n=1}^\infty|(\bm u_0,\bm\vp_n)|^2=\left(C\|\bm u_0\|_{L^2(\Om)} t^{-\al\ga}\right)^2.
\end{align*}
This implies that \eqref{eq-solest1} holds true, and uniqueness can be directly obtained by setting $\bm u_0=\bm0$ (see \cite{SY11} for details). Therefore, it is straightforward to deduce \eqref{eq-solest2} by
\[
\|\bm u\|_{L^{1/\ga}(0,T;D(\cA^\ga))}=\left(\int_0^T\|\bm u(t)\|_{D(\cA^\ga)}^{1/\ga}\,\rd t\right)^\ga
%\le C\|\bm u_0\|_{L^2(\Om)}\left(\int_0^T t^{-\al}\,\rd t\right)^\ga
\le C\left(\f{T^{1-\al}}{1-\al}\right)^\ga\|\bm u_0\|_{L^2(\Om)}.
\]
Especially, setting $\ga=1/2$ in \eqref{eq-solest2} yields
\[
\bm u\in L^2(0,T;D(\cA^{1/2}))=L^2(0,T;(H_0^1(\Om))^K).
\]
Since $\cA$ is a second order elliptic operator with bounded coefficients, the governing equation in \eqref{eq-IBVP0} indicates that $\pa_t^\al(\bm u-\bm u_0)=-\cA\bm u\in L^2(0,T;(H^{-1}(\Om))^K)$. Thus, we see that $\bm u-\bm u_0$ belongs to $H_\al(0,T;(H^{-1}(\Om))^K)$ and shares the same estimate as \eqref{eq-solest2}. This completes the proof of (i).\medskip

(ii) We will show that \eqref{eq-solrep2} certainly gives the solution to \eqref{eq-IBVP0} when $\bm u_0\equiv\bm0$. Taking into account Lemma \ref{lem-ML}(ii)--(iii), similar to \cite{SY11}, we can derive
\begin{equation}\label{eq-mitint}
\int_0^\eta\left|t^{\al-1}E_{\al,\al}(-\la_n t^\al)\right|\rd t=\f{1-E_{\al,1}(-\la_n \eta^\al)}{\la_n}\le\f1{\la_n},\quad\forall\,\eta>0.
\end{equation}
Meanwhile, direct calculations yield
\begin{align}
& \quad\,\pa_t^\al\int_0^t s^{\al-1}E_{\al,\al}(-\la_n s^\al)(\bm F(\,\cdot\,,t-s),\bm\vp_n)\,\rd s\nonumber\\
& =-\la_n\int_0^t s^{\al-1}E_{\al,\al}(-\la_n s^\al)(\bm F(\,\cdot\,,t-s),\bm\vp_n)\,\rd s+(\bm F(\,\cdot\,,t),\bm\vp_n).\label{eq-mitlap}
\end{align}
Utilizing \eqref{eq-mitint}--\eqref{eq-mitlap} and applying Young's inequality for convolutions, we obtain
\begin{align*}
& \quad\,\left\|\pa_t^\al\int_0^t s^{\al-1}E_{\al,\al}(-\la_n s^\al)(\bm F(\,\cdot\,,t-s),\bm\vp_n)\,\rd s\right\|_{L^2(0,T)}^2\\
& \le C\int_0^T(\bm |\bm F(\,\cdot\,,t),\bm\vp_n)|^2\,\rd t+C\left(\int_0^T|(\bm F(\,\cdot\,,t),\bm\vp_n)|^2\,\rd t\right)\left(\int_0^T\left|\la_n t^{\al-1} E_{\al,\al}(-\la_n t^\al)\right|\rd t\right)^2\\
& \le C\int_0^T|(\bm F(\,\cdot\,,t),\bm\vp_n)|^2\,\rd t
\end{align*}
and therefore
\begin{align*}
\|\bm u\|^2_{H_\al(0,T;(L^2(\Om))^K)} & =\sum_{n=1}^\infty\int_0^T\left|\pa_t^\al\int_0^t s^{\al-1}E_{\al,\al}(-\la_n s^\al)(\bm F(\,\cdot\,,t-s),\bm\vp_n)\,\rd s\right|^2\rd t\\
& \le C\int_0^T|(\bm F(\,\cdot\,,t),\bm\vp_n)|^2\,\rd t=C\|\bm F\|^2_{L^2(0,T;(L^2(\Om))^K)}.
\end{align*}
Again by the governing equation in \eqref{eq-IBVP0}, we see $\cA\bm u=\bm F-\pa_t^\al\bm u$, which means
\[
\|\cA\bm u\|_{L^2(0,T;(L^2(\Om))^K)}\le C\|\bm F\|_{L^2(0,T;(L^2(\Om))^K)}
\]
and thus implies \eqref{eq-solest4}. Finally, the uniqueness can be similarly obtained as in \cite{SY11}. We have completed the proof of Theorem \ref{thm-FP}.
\end{proof}

%%%%%proof of Theorem 3
\begin{proof}[Proof of Theorem $\ref{thm-BP}$]
According to \eqref{eq-solrep1}, we have
\[
\bm u_1 :=\bm u(\,\cdot\,,T)=\sum_{n=1}^\infty E_{\al,1}(-\la_n T^\al)(\bm u_0,\bm\vp_n)\bm\vp_n.
\]
By $\bm u_1\in(H^2(\Om)\cap H_0^1(\Om))^K$ and the norm equivalence between $H^2(\Om)\cap H_0^1(\Om)$ and $D(\cA)$, there exist constants $C_2'>C_1'>0$ such that
\[
C_1'\|\bm u_1\|_{H^2(\Om)}^2\le\sum_{n=1}^\infty|\la_n(\bm u_1,\bm\vp_n)|^2\le C_2'\|\bm u_1\|_{H^2(\Om)}^2.
\]
By Lemma \ref{lem-ML}(iii), we have $E_{\al,1}(-\la_n t^\al)>0$ for any $t>0$. Then we can reconstruct the initial value by
\[
\bm u_0=\sum_{n=1}^\infty\f{(\bm u_1,\bm\vp_n)}{E_{\al,1}(-\la_n T^\al)}\bm\vp_n.
\]
As a result, it follows from Lemma \ref{lem-ML}(iv) that
\begin{align*}
\|\bm u_0\|_{L^2(\Om)}^2 &
%=\sum_{n=1}^\infty c_n^2
=\sum_{n=1}^\infty\left(\f{(\bm u_1,\bm\vp_n)}{E_{\al,1}(-\la_n T^\al)}\right)^2=\sum_{n=1}^\infty\left(\f{\la_n T^\al\Ga(1-\al)(\bm u_1,\bm\vp_n)}{1+O(\la_n^{-1}T^{-\al})}\right)^2\\
& \le(C T^\al)^2\sum_{n=1}^\infty|\la_n(\bm u_1,\bm\vp_n)|^2\le\left(C C_2'T^\al\|\bm u_1\|_{H^2(\Om)}\right)^2,
\end{align*}
which completes the proof of Theorem \ref{thm-BP}.
\end{proof}

%%%%%%%%%%%%%%%%%%%%%%%%%%%%%%%%%%%%%%%%

\section{Inversion Algorithm and Numerical Results}\label{sec-numer}

In the following, we present comprehensive numerical results pertaining to the aforementioned backward problem. We firstly outline the precise architecture of the neural network utilized in our study. Subsequently, we demonstrate one-dimensional and two-dimensional numerical examples to illustrate the effectiveness of our inversion algorithm. Beyond merely demonstrating its efficacy, we also provide partial validation of the Lipschitz stability, as stated in Theorem \ref{thm-BP}, thereby further strengthening the reliability of our approach.

\subsection{Neural network architecture}

The proposed backward problem involves reconstructing the initial condition $\bm u_0$ based on the final time measurement $\bm u(x,T)$ where $x\in \Omega$. While traditional approaches rely on standard regularization schemes to address this inverse problem\cite{LY10,RXL14}, we aim to harness the potential of neural networks in solving such challenges. However, due to the inherent ill-posed and nonlocal nature of the backward problem, directly applying fractional physics-informed neural networks or other neural network architectures often leads to ineffective algorithms for the inverse problem of the coupled systems. Therefore, we will develop a novel neural network architecture designed specifically to solve this inverse problem.
Let us first describe the standard inversion algorithm by recalling the solution expansion form \eqref{eq-solrep1} and set $t=T$ such that
\[
\bm u(\,\cdot\,,T)=\sum_{n=1}^\infty E_{\al,1}(-\la_n T^\al)(\bm u_0,\bm\vp_n)\bm\vp_n.
\]
We can observe that the function values at the final time are expressed as a linear combination of basis functions $\bm \vp_n$. Furthermore, the coefficients of these basis functions consist of the projection of the initial value $\bm{u}_0$ onto the basis functions, as well as the Mittag-Leffler functions $E_{\alpha,1}(-\lambda_n T^\alpha)$. It is important to note that once the diffusion coefficients of the equation, specifically $\mathcal{A}$ in \eqref{eq-IBVP0}, are fixed, $\lambda_n$ and $\bm \varphi_n$ are fixed.
It is straightforward to observe that the final time value $\bm{u}_T$ is determined by $\bm{u}_0$. If we denote the mapping from initial values $\bm u_0$ to $\bm u_T$ as $\mathcal{K}$, then we have
$$
\bm{u}_T = \mathcal{K}\bm{u}_0.
$$
Here, $\mathcal{K}$ is the forward operator and includes the information of the coupled equations, and our objective is using the neural networks to learn information of this mapping. For extended discussion on operator learning, we refer to \cite{dL2022, dK2023}, as well as the review article \cite{AM2019}. Once we have established the mapping from initial values to final values, for the backward problem, the standard Tikhonov regularization method can be expressed as follows:
\begin{align}\label{eq-BackwardTikhonov}
\bm{u}_0 \approx   (\mathcal{K}^T \mathcal{K}+\epsilon I)^{-1}\mathcal{K}^T \bm{u}_T
\end{align} by introducing a regularization parameter $\epsilon$.

We employ a multi-channel neural network to approximate the mapping $(\mathcal{K}^T \mathcal{K}+\epsilon I)^{-1}\mathcal{K}^T$ in \eqref{eq-BackwardTikhonov}, and we illustrate the network architecture depicted in Figure \ref{Fig:inverse_architecture} using a 2-dimensional case. Specifically, we design two neural networks to approximate the mappings $\mathcal{K}^T$ and $(\mathcal{K}^T \mathcal{K}+\epsilon I)^{-1}$.
For the first neural network, approximating $\mathcal{K}^T$, we consider a coupled system with $K$ components. Each discrete final time value is represented as a matrix of size $N_x \times N_y$, forming an input tensor of dimensions $N_x \times N_y \times K$. After passing through a Convolutional Layer (CNN), we obtain an intermediate tensor of dimensions $N_n \times N_n \times K$, where $N_n$ is determined by the size of the convolutional kernel. This layer establishes connections between the final values, enabling the network to capture interactions between different diffusion terms. Next, a Split Layer divides the intermediate tensor into $K$ separate tensors, each of size $N_n \times N_n$. These tensors then propagate through several Fully Connected Layers (FC) within each channel, allowing the network to learn the diffusion process from initial to final values. A Concatenate Layer combines the outputs from all channels, resulting in a tensor of dimensions $N_m \times N_m \times K$. However, this tensor is not our final output for the backward mapping.
To approximate the second mapping $(\mathcal{K}^T \mathcal{K}+\epsilon I)^{-1}$, we introduce another Convolutional Layer at the end. Instead of manually setting the regularization parameter $\epsilon$, we allow the network to adaptively learn this parameter during training. The process for other dimensional spaces is similar, requiring only adjustments to the data dimensions.
\begin{figure}[htbp]
\centering
\includegraphics[width=1\linewidth]{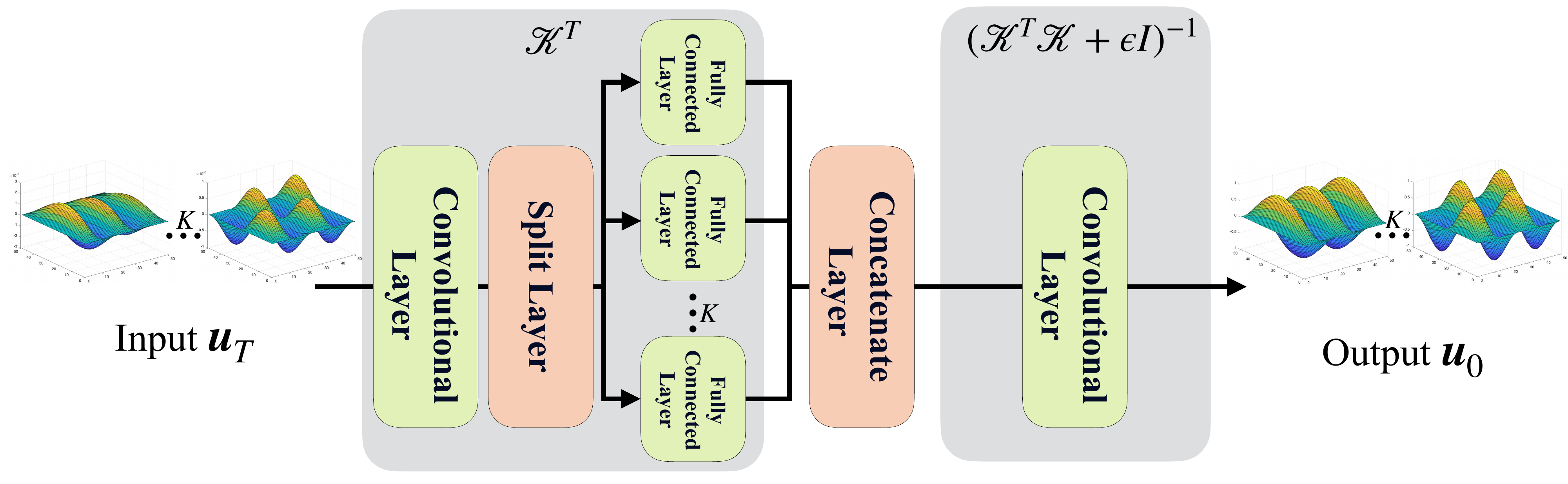}
\caption{Multi-channel neural network architecture for the backward problem.}
\label{Fig:inverse_architecture}
\end{figure}

The neural network is implemented using Keras, which acts as a high-level API for TensorFlow. We employ the Nadam optimizer\cite{D2016} with a learning rate set to  $10^{-3}$ and adopt the mean squared error in the loss function. Specifically, we configure the fully connected (FC) network with two layers, using the tanh activation function and the sinusoid activation function for the two layers, respectively. Given this neural network architecture, we can utilize traditional numerical schemes, such as the finite difference methods in \cite{JZ23, LX2007}, to generate the requisite data pairs for training the network.

\subsection{Numerical experiments}

In this subsection, we set $K=2$ in Figure \ref{Fig:inverse_architecture} and investigate the following coupled equations
\begin{equation}\label{Eq:lin_sys_num}
\begin{cases}
\!\begin{aligned}
& \pa_t^\al(u-u_0)-\rdiv(\bm A(\bm x)\nabla u)+c_{11}u+c_{12}v=0\\
& \pa_t^\al(v-v_0)-\rdiv(\bm B(\bm x)\nabla v)+c_{21}u+c_{22}v=0
\end{aligned} & \mbox{in }\Om\times(0,T),\\
u=v=0 & \mbox{on }\pa\Om\times(0,T).
\end{cases}
\end{equation}
We present numerical examples for both one-dimensional and two-dimensional scenarios,  assuming $c_{11}=c_{22}=1$ and $c_{12}=c_{21}=-1$ in the coupled system given by \eqref{Eq:lin_sys_num}. For the sake of simplicity and convenience in programming,  in the one-dimensional case, we set
\[
A(x)=B(x)=1+x,
\]
and in the two-dimensional case, we set
\[
\bm A(x,y)=\begin{pmatrix}
1 + x^2 + y^2 & 0 \\
0 & 1 + x^2 + y^2
\end{pmatrix},\ \bm B(x,y)=\begin{pmatrix}
3 + \cos{x} + \cos{y} & 0 \\
0 & 3 + \cos{x} + \cos{y}
\end{pmatrix}.
\]

\subsubsection*{One-diminsonal case}

\begin{figure}[htbp]
\centering
\begin{subfigure}[b]{0.4\linewidth}
\includegraphics[width=\linewidth]{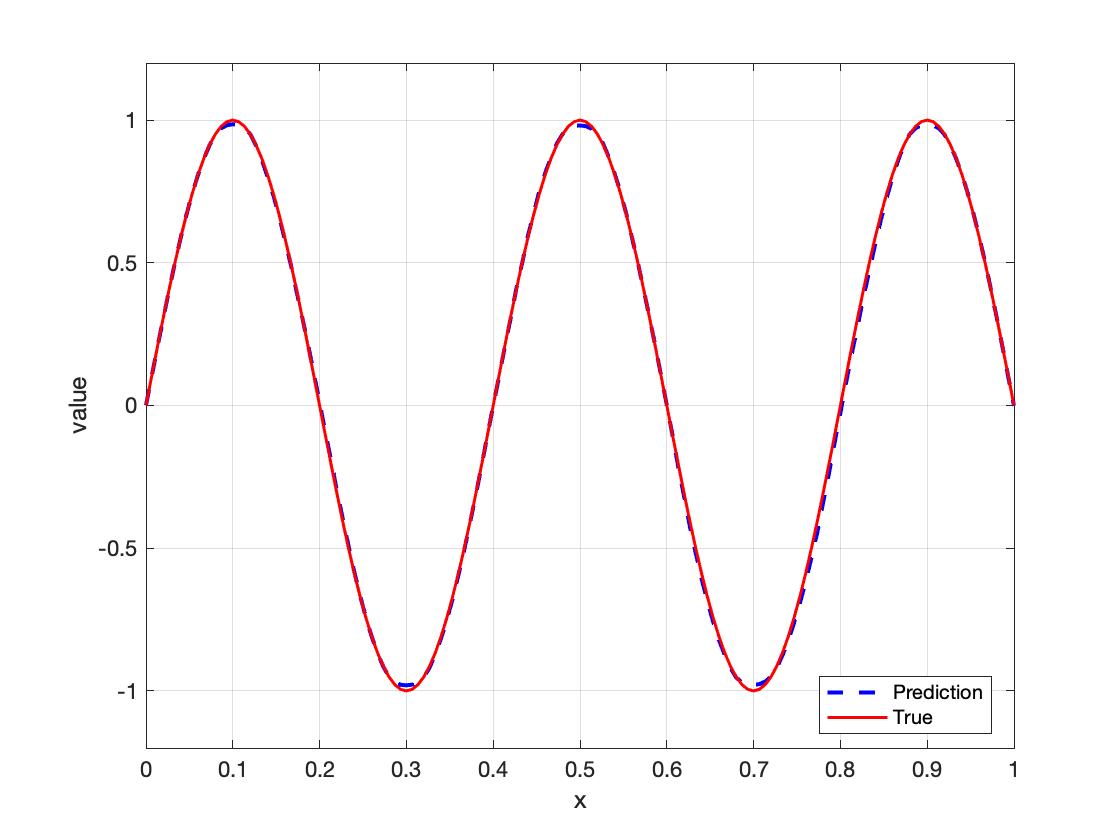}
\caption{Reconstruction with noise-free data}
\end{subfigure}
\begin{subfigure}[b]{0.4\linewidth}
\includegraphics[width=\linewidth]{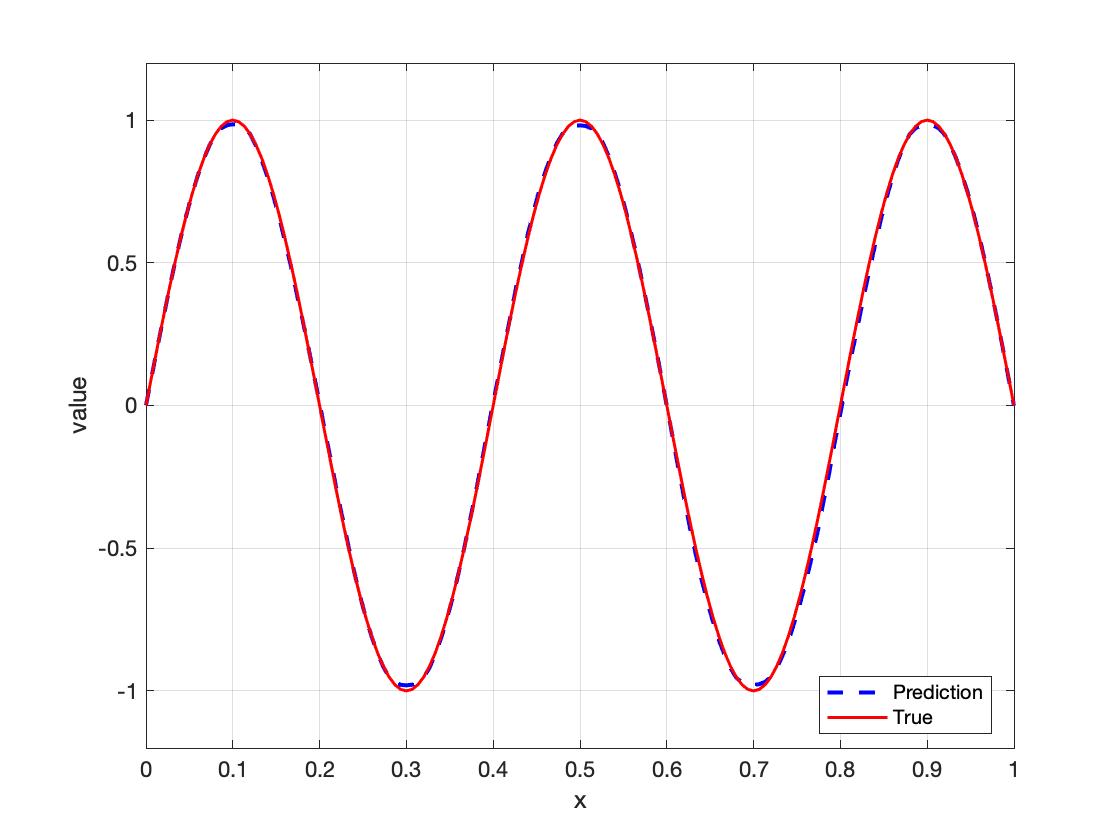}
\caption{Reconstruction with  1\% noise}
\end{subfigure}
\\
\begin{subfigure}[b]{0.4\linewidth}
\includegraphics[width=\linewidth]{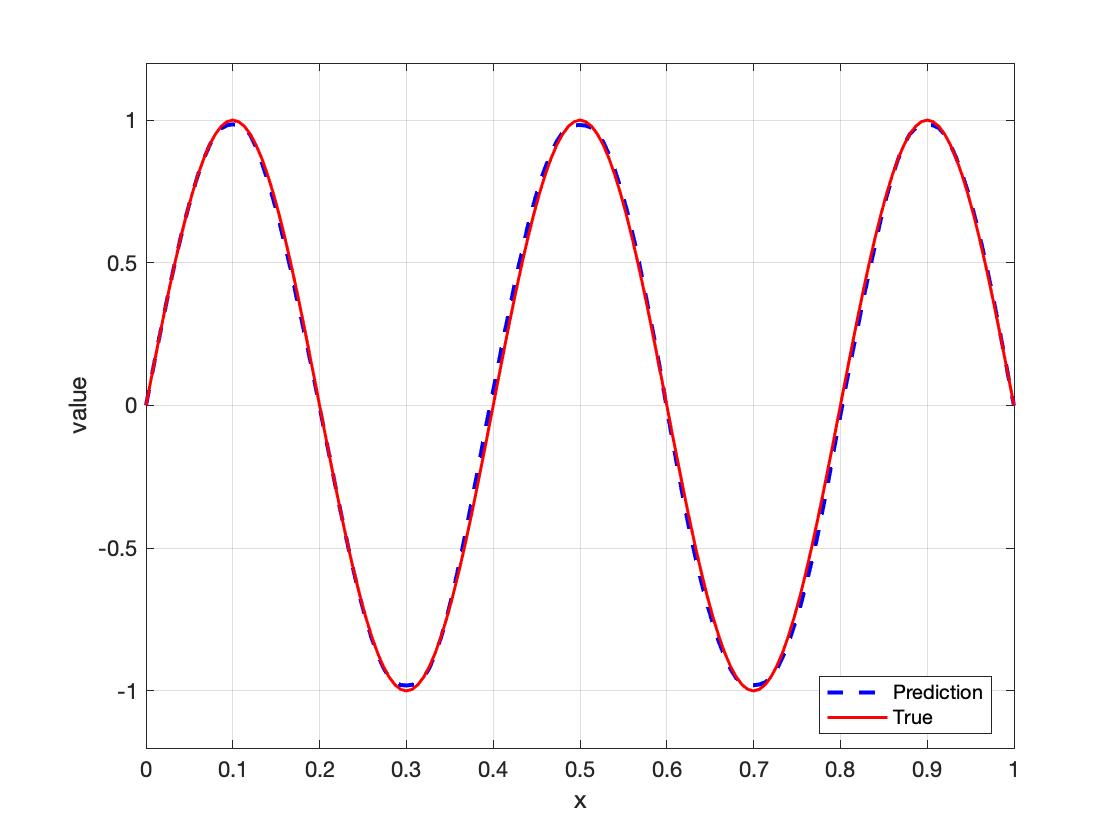}
\caption{Reconstruction with 5\% noise}
\end{subfigure}
\begin{subfigure}[b]{0.4\linewidth}
\includegraphics[width=\linewidth]{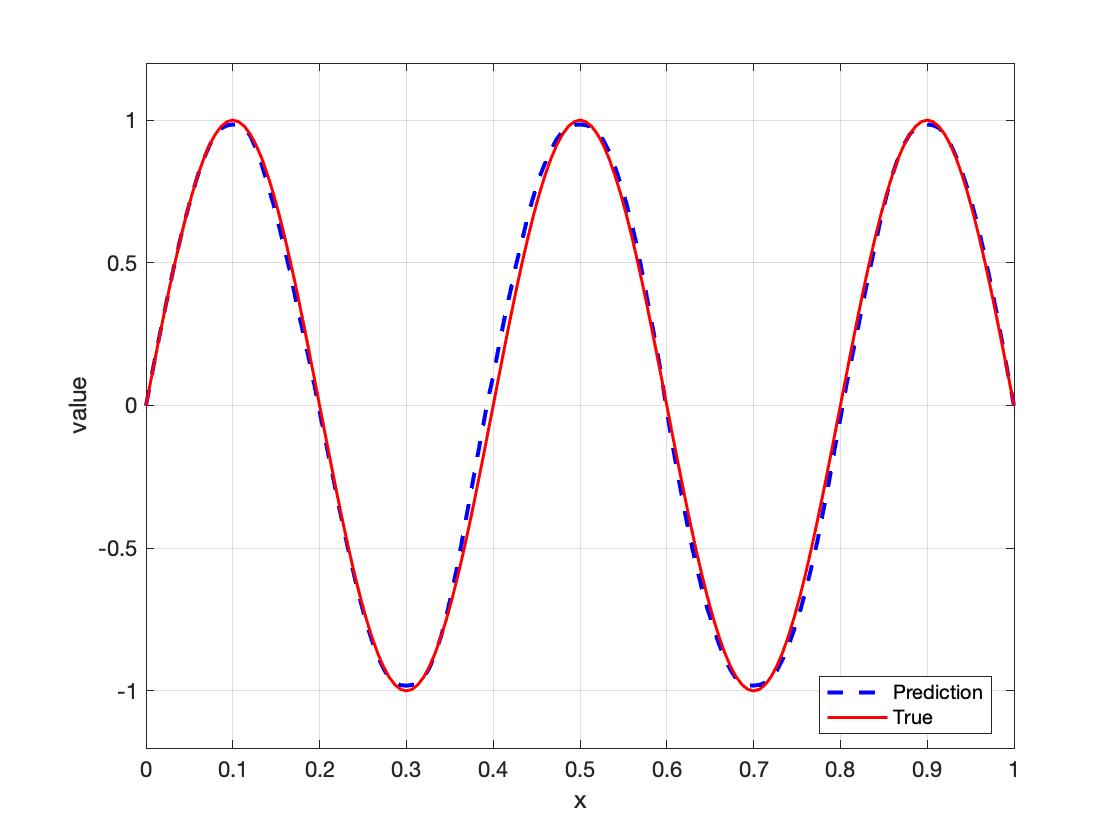}
\caption{Reconstruction with 10\% noise}
\end{subfigure}
\caption{True and reconstructed solutions for $\alpha=0.2$ and $T=1$ with noise-free and $1\%$, $5\%$, $10\%$ relative-noise measurements. In all panels, the red solid line is the true initial value $u_0(x)=\sin(5\pi x)$ and the blue dashed line is the reconstructed initial value, while $v_0(x)=\sin(3\pi x)$.}
\label{Fig:s23a2}
\end{figure}

\begin{figure}[htbp]
\centering
\begin{subfigure}[b]{0.4\linewidth}
\includegraphics[width=\linewidth]{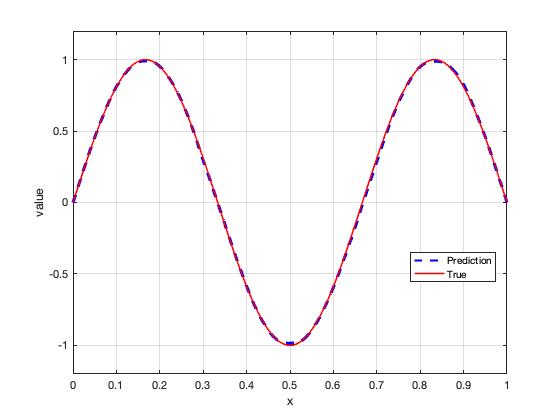}
\caption{Reconstruction with noise-free data}
\end{subfigure}
\begin{subfigure}[b]{0.4\linewidth}
\includegraphics[width=\linewidth]{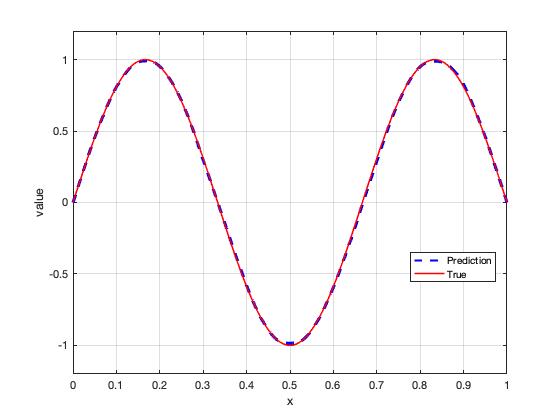}
\caption{Reconstruction with  1\% noise}
\end{subfigure}
\\
\begin{subfigure}[b]{0.4\linewidth}
\includegraphics[width=\linewidth]{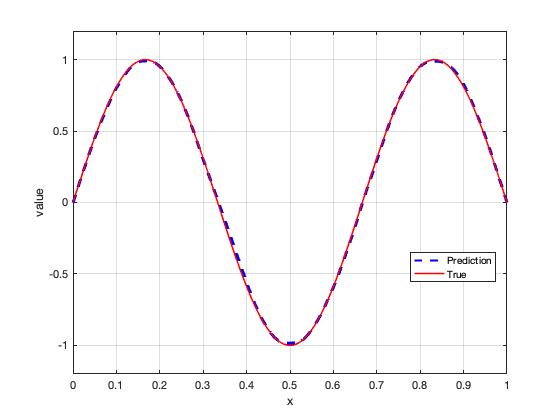}
\caption{Reconstruction with  5\% noise}
\end{subfigure}
\begin{subfigure}[b]{0.4\linewidth}
\includegraphics[width=\linewidth]{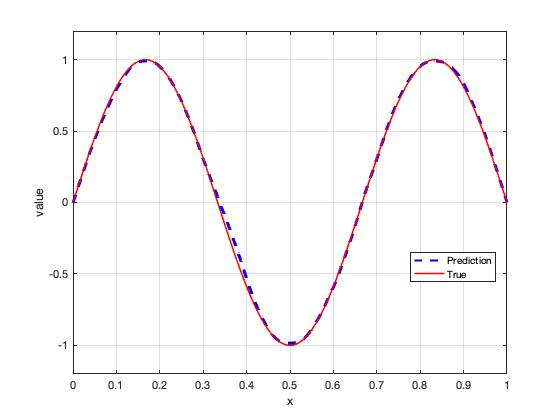}
\caption{Reconstruction with 10\% noise}
\end{subfigure}
\caption{True and reconstructed solutions for $\alpha=0.2$ and $T=1$, with noise-free and $1\%$, $5\%$, $10\%$ relative-noise measurements. In all panels, the red solid line is the true initial value $u_0(x)=\sin(3\pi x)$ and the blue dashed line is the reconstructed initial value,  while $v_0(x)=\sin(4\pi x)$.}
\label{Fig:s14a2}
\end{figure}

In the one-dimensional scenario, we begin with initial values $(u_0,v_0)$  and proceed to compute the corresponding final values $(u_T,v_T)$ at a specified time $T$. We select
$$u_0(x)=\sin(i\pi x), \quad v_0(x)=\sin(j\pi x),$$
where $i,j=1,\cdots,5$, and generate the corresponding $(u_T,v_T)$. This results in 25 sets of $(u_T, v_T)$ and $(u_0, v_0)$ pairs, which we use as training data. We then input these pairs into the neural network depicted in Figure \ref{Fig:inverse_architecture}. It is important to note that, in the one-dimensional context, the tensor shown in Figure \ref{Fig:inverse_architecture} shall be adapted to a suitable second-order tensor.

We set the spatial discretization $N_x=150$ and compute the numerical solution at $T=1$ as the measurement data $(u_T,v_T)$.  Figures \ref{Fig:s23a2} and \ref{Fig:s14a2} present the numerical results obtained for the backward problem, specifically for the unknown solutions $u_0(x)= \sin(5\pi x), v_0(x)= \sin(3\pi x)$ and $u_0(x)= \sin(3\pi x), v_0(x)= \sin(4\pi x)$, respectively.
It is evident that when the unknown solution is included in the training set, the trained neural network is able to achieve a remarkably accurate reconstruction. Even under noisy observation conditions, the neural network exhibits considerable robustness. As the noise level increases, the quality of the initial value reconstruction may decrease somewhat, but the neural network still does well in finding the peak positions and capturing the overall trend of the initial values. Here we note that all noisy measurements are represented by $\boldsymbol{u}_T^{\delta}:=\boldsymbol{u}_T+ \delta \xi$, where $\xi$ is a normalized Gaussian random variable with zero mean and unit variance. The parameter $\delta$ controls the level of noise, and we select relative noise levels of $\delta=1\%, 5\%$, and $10\%$ to evaluate the performance under different noise conditions.

\begin{figure}[htbp]
\centering
\begin{minipage}{0.45\textwidth}
\centering    \includegraphics[width=\linewidth]{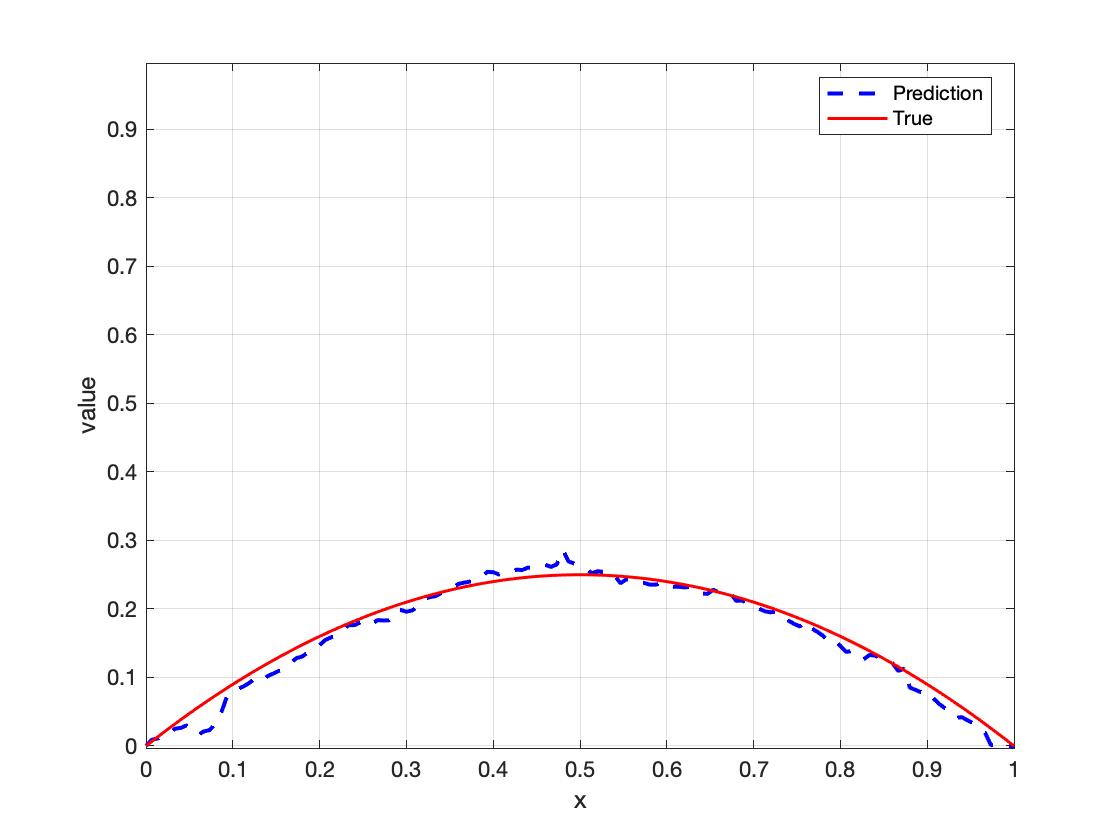}
\end{minipage}
\hfill
\begin{minipage}{0.45\textwidth}
\centering
\includegraphics[width=\linewidth]{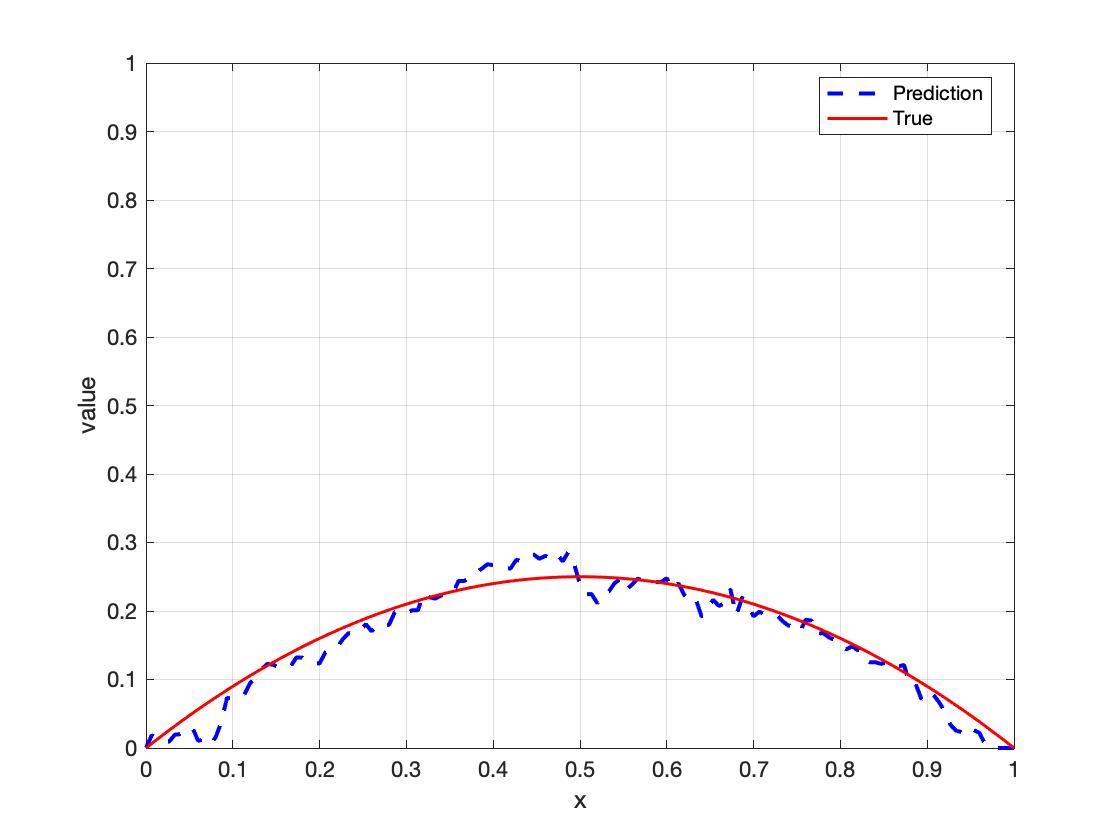}
\end{minipage}
\caption{True and reconstructed solutions for $\alpha=0.2$ and $T=1$, with noise-free and $10\%$ relative-noise measurements. In both panels, the red solid line is the true initial value $u_0(x)=v_0(x)=x(1-x)$ and the blue dashed line is the reconstructed initial value.}
\label{Fig:x_1minusx}
\end{figure}
The following numerical example demonstrates that the trained neural network has good generalization capabilities, rather than just memorizing the training data set. To validate this, we test the neural network with final time measurements that are not part of the original training set. Remarkably, it still produces excellent inversion results. Specifically, using the classical finite difference method, we compute the final time measurement data $u_T$ and $v_T$ at time $T=1$ for the initial values $u_0(x) = v_0(x) = x(1-x)$. We then use these measurement data as input for the trained neural network. In Figure \ref{Fig:x_1minusx}, we show the reconstructed results of the neural network for the initial values by noise-free and noisy measurement with a relative noise level of $10\%$.

\subsubsection*{Two-dimensional case}

\begin{figure}[htbp]\centering
\begin{subfigure}[b]{0.4\linewidth}
\includegraphics[width=\linewidth]{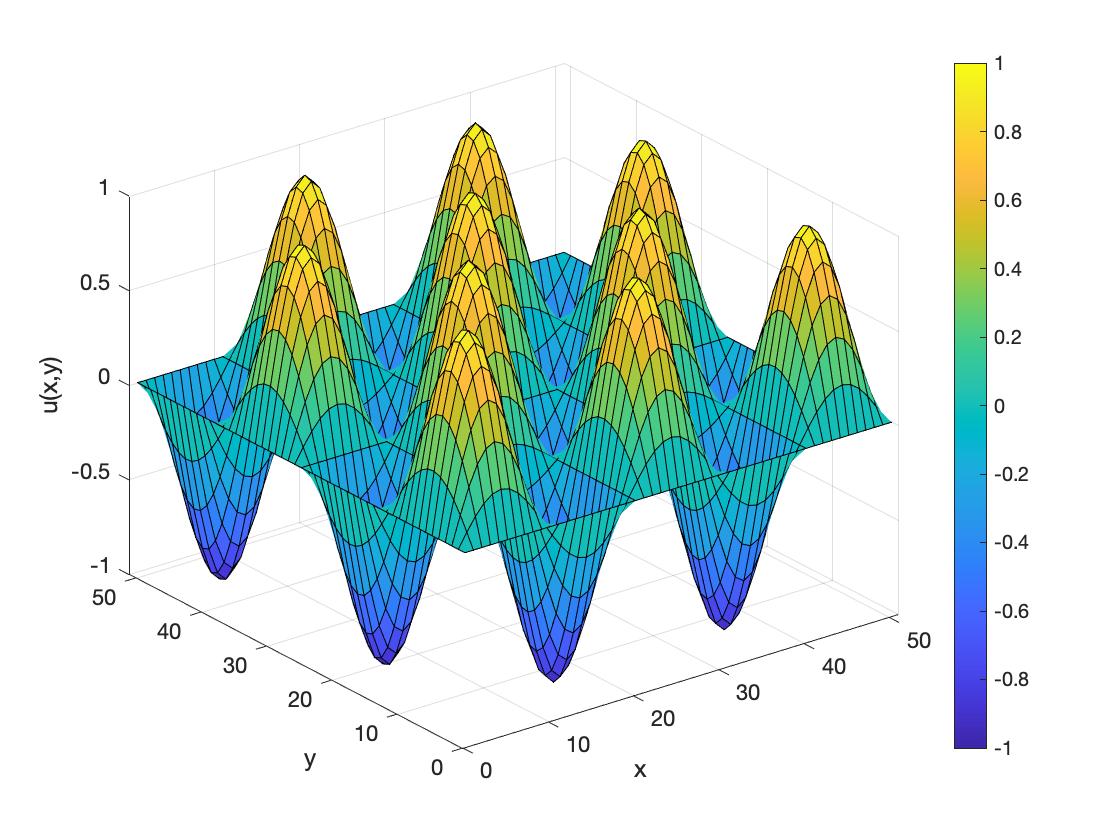}
\caption{True initial value}
\end{subfigure}
\begin{subfigure}[b]{0.4\linewidth}
\includegraphics[width=\linewidth]{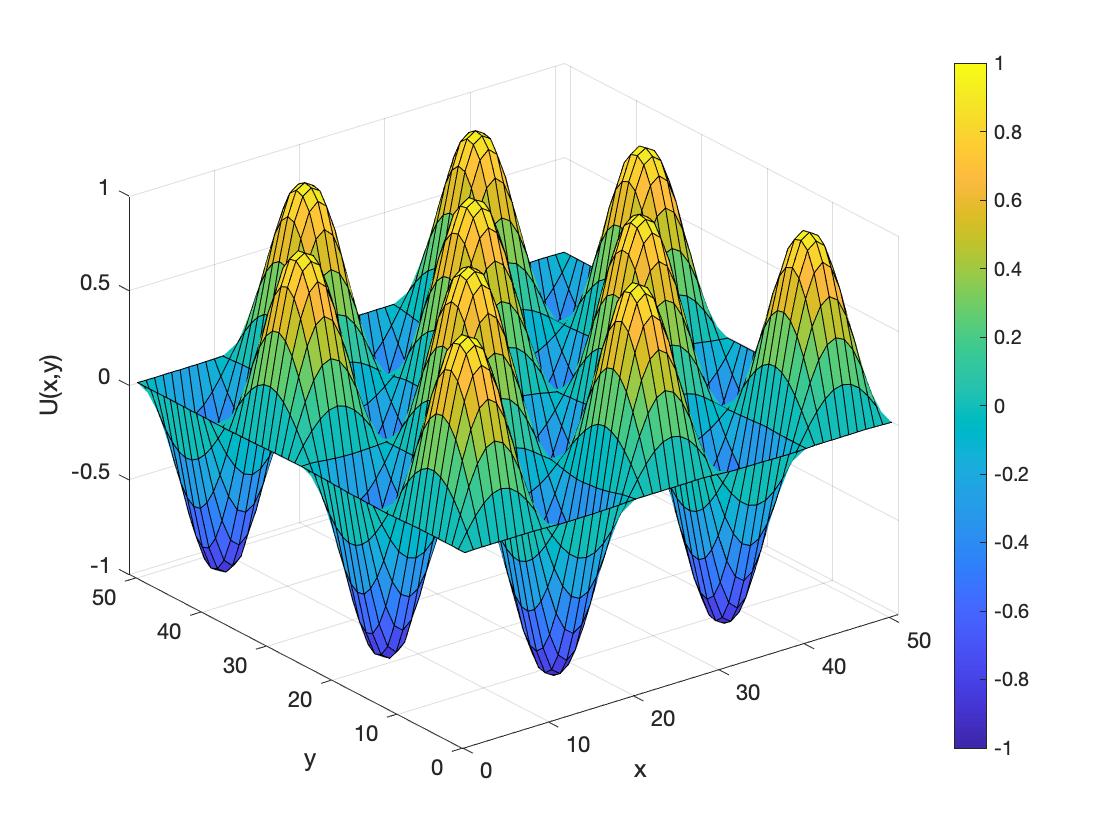}
\caption{Reconstruction with noise-free data}
\end{subfigure}\\
\begin{subfigure}[b]{0.4\linewidth}
\includegraphics[width=\linewidth]{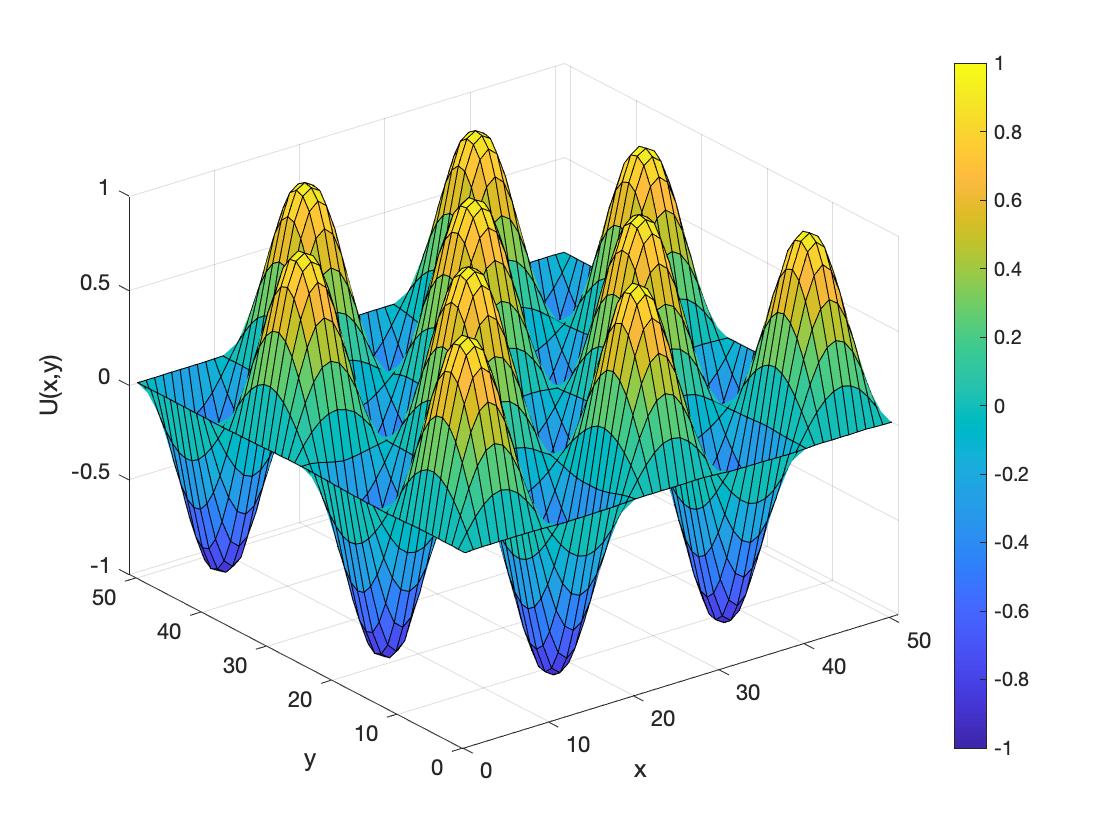}
\caption{Reconstruction with 5\% noise}
\end{subfigure}
\begin{subfigure}[b]{0.4\linewidth}
\includegraphics[width=\linewidth]{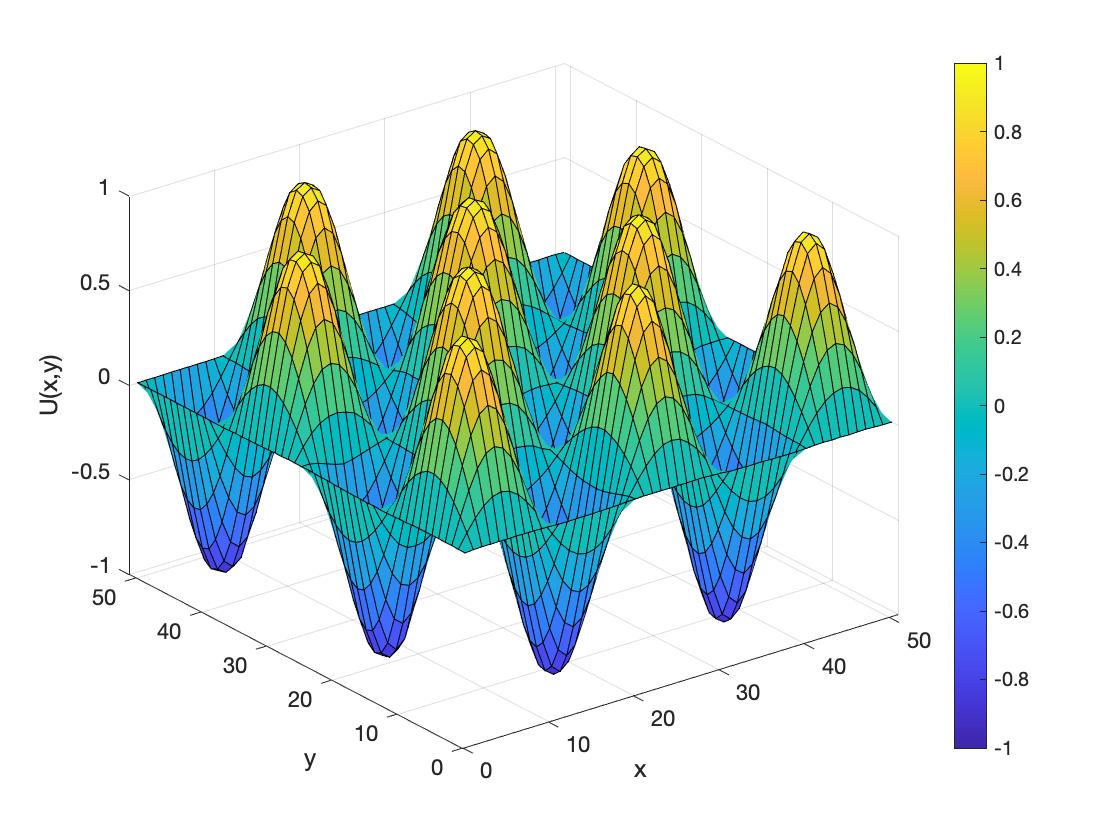}
\caption{Reconstruction with 10\% noise}
\end{subfigure}
\caption{True and reconstructed solutions for $\alpha=0.4$ and $T=1$, with noise-free and $5\%$, $10\%$ relative-noise measurements, where $u_0(x,y)=\sin(5\pi x)\sin(4\pi y), v_0(x,y)=\sin(\pi x)\sin(3\pi y)$.}
\label{Fig:s8a4}
\end{figure}

We next present several two-dimensional numerical examples. In the two-dimensional case, the training data is selected as follows. We establish the training initial data as
$$
u_0(x,y)=\sin(i\pi x)\sin(j\pi y), \quad v_0(x,y)=\sin(k\pi x)\sin(\ell\pi y),
$$
where $i,j,k,\ell=1,\cdots,5$, to generate a total of 625 pairs of final measurement of $(u_T,v_T)$ by the finite difference method. We then utilize these pairs as the training data and input them into the neural network to learn the mapping from the final measurements to the initial value.
Similar to the previous subsection, we first select one sample from the training data set to demonstrate the neural network's effectiveness, for instance by $u_0(x,y)=\sin(5\pi x)\sin(4\pi y)$ and $v_0(x,y)=\sin(\pi x)\sin(3\pi y)$.
Figure \ref{Fig:s8a4} illustrates the reconstructed initial values for $u_0$ under different measurements: with noise-free and $5\%$, $10\%$ relative-noise measurements.

\begin{figure}[htbp]\centering
\begin{subfigure}[b]{0.4\linewidth}
\includegraphics[width=\linewidth]{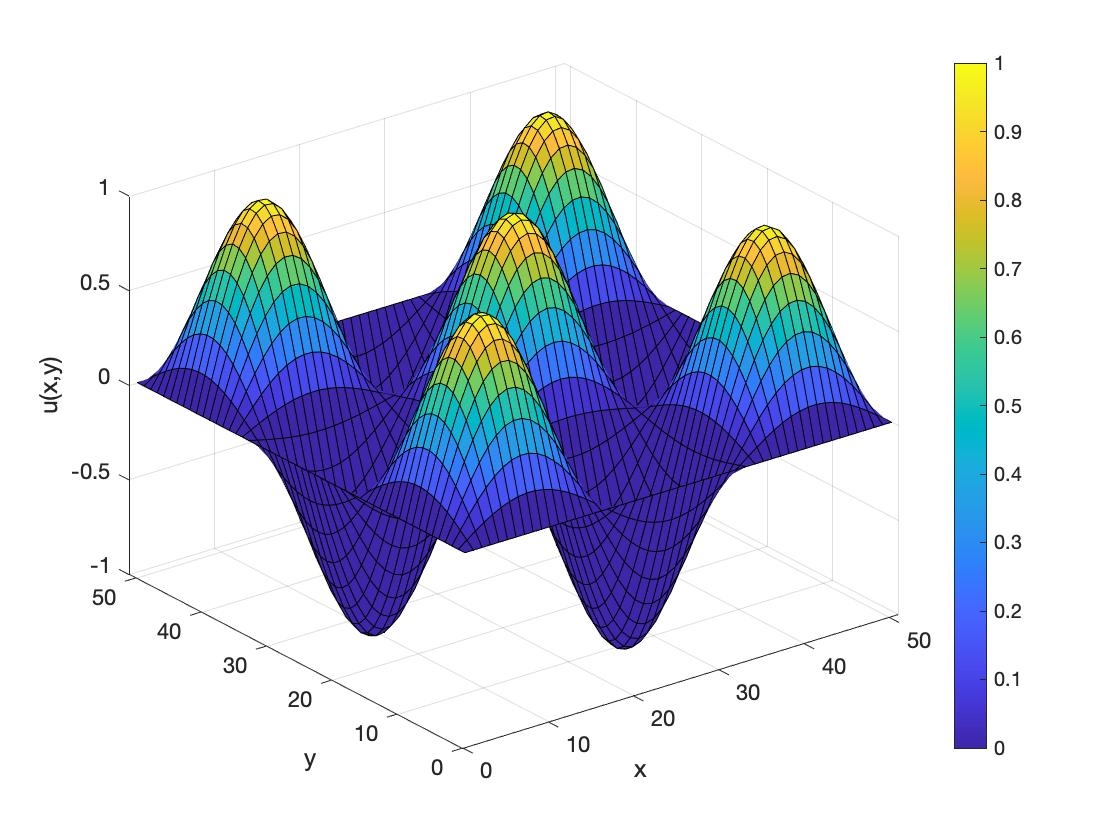}
\caption{True initial value $u_0$}
\end{subfigure}
\begin{subfigure}[b]{0.4\linewidth}
\includegraphics[width=\linewidth]{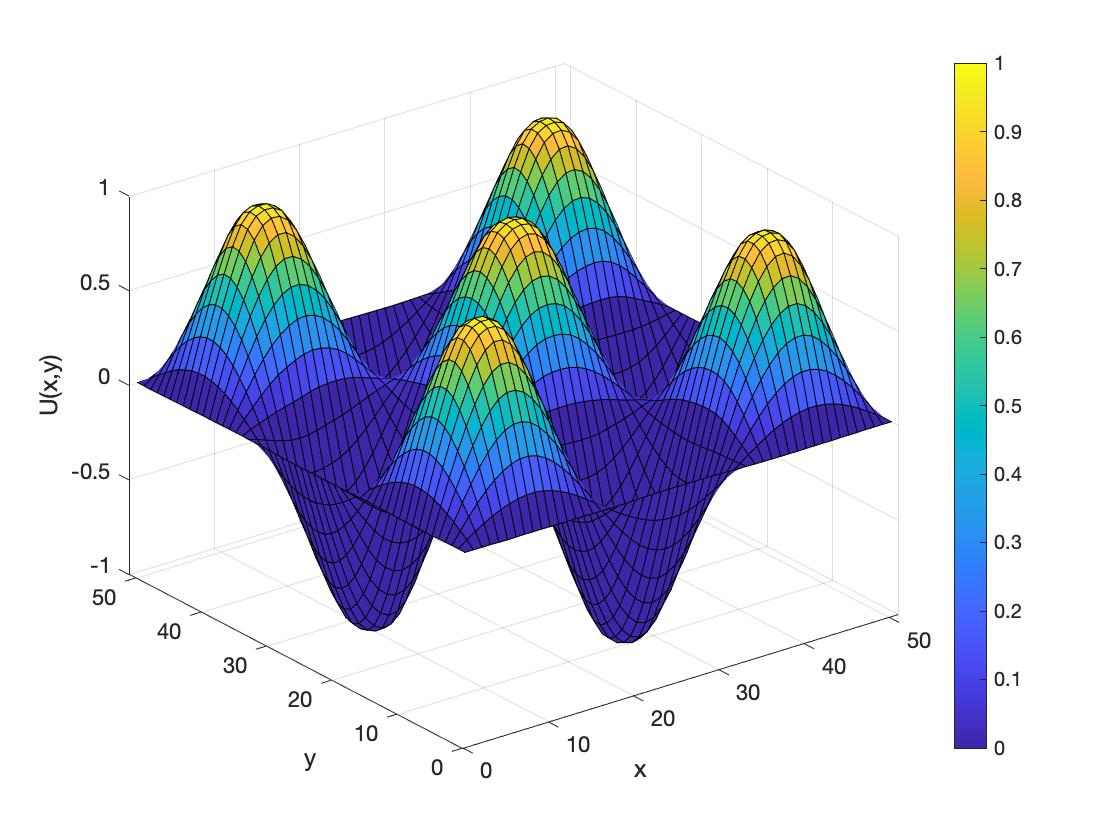}
\caption{Reconstructed $u_0$}
\end{subfigure}\\
\begin{subfigure}[b]{0.4\linewidth}
\includegraphics[width=\linewidth]{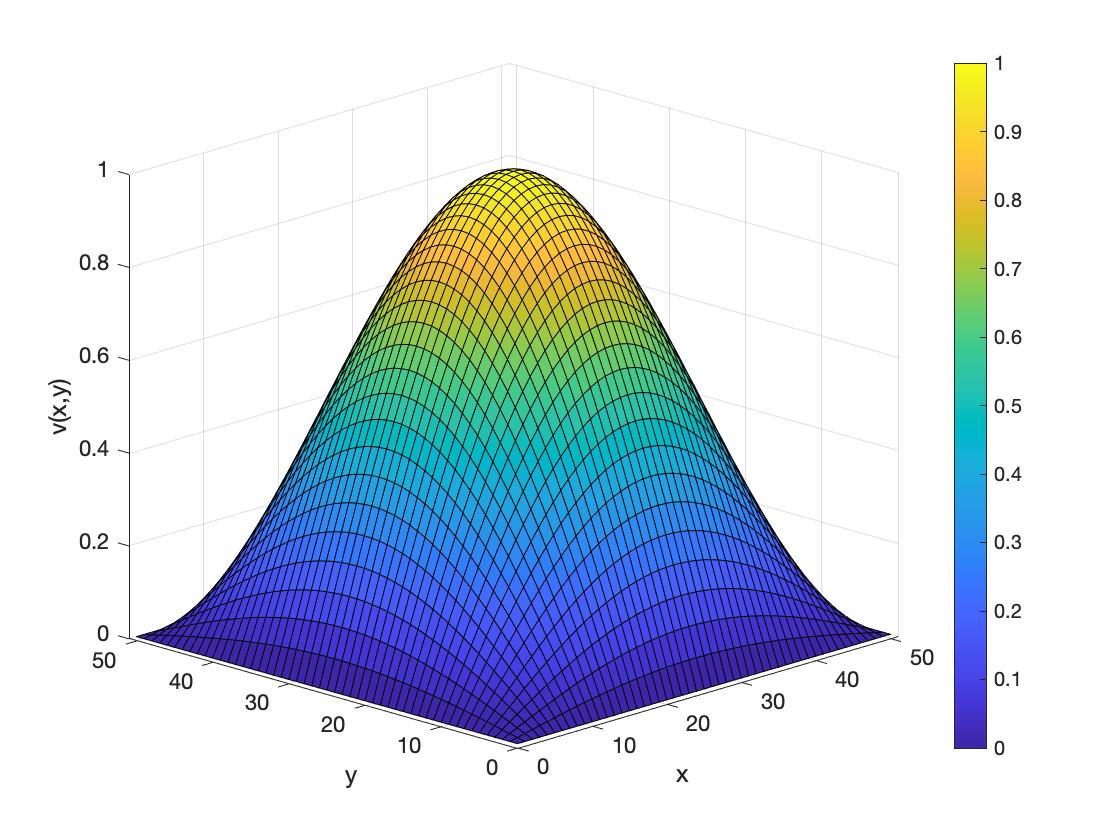}
\caption{True initial value $v_0$}
\end{subfigure}
\begin{subfigure}[b]{0.4\linewidth}
\includegraphics[width=\linewidth]{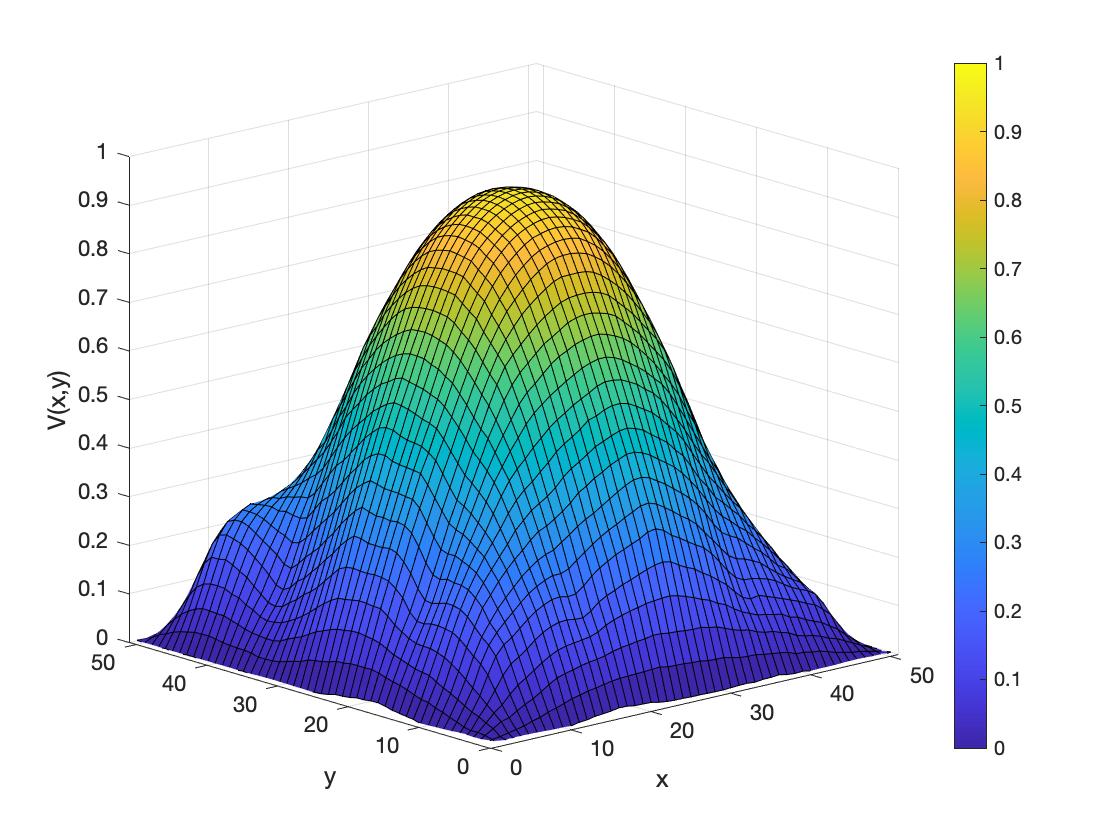}
\caption{Reconstructed $v_0$}
\end{subfigure}
\caption{True and reconstructed solutions for $\alpha=0.4$ and $T=1$ with $10\%$ relative-noise measurement, where $u_0(x,y)=\sin(3\pi x)\sin(3\pi y)$ and $v_0(x,y)=2^8 x^2(1-x)^2y^2(1-y)^2$.}
\label{Fig:xy_1minusxy}
\end{figure}

To validate the generalization capabilities of the trained neural network, we display in Figure \ref{Fig:xy_1minusxy} the reconstructed result under $10\%$ relative-noise measurement for the initial value $u_0(x,y) = \sin(3\pi x)\sin(3\pi y)$ and $v_0(x,y)=2^8 x^2(1-x)^2y^2(1-y)^2$, where the latter $v_0$ is not a part of the original training set. Similar to the 1-D case, the trained neural network exhibits its generalization property and provides a reasonable reconstruction of the initial value.
\begin{figure}[htbp]\centering
\includegraphics[width=0.5\linewidth]{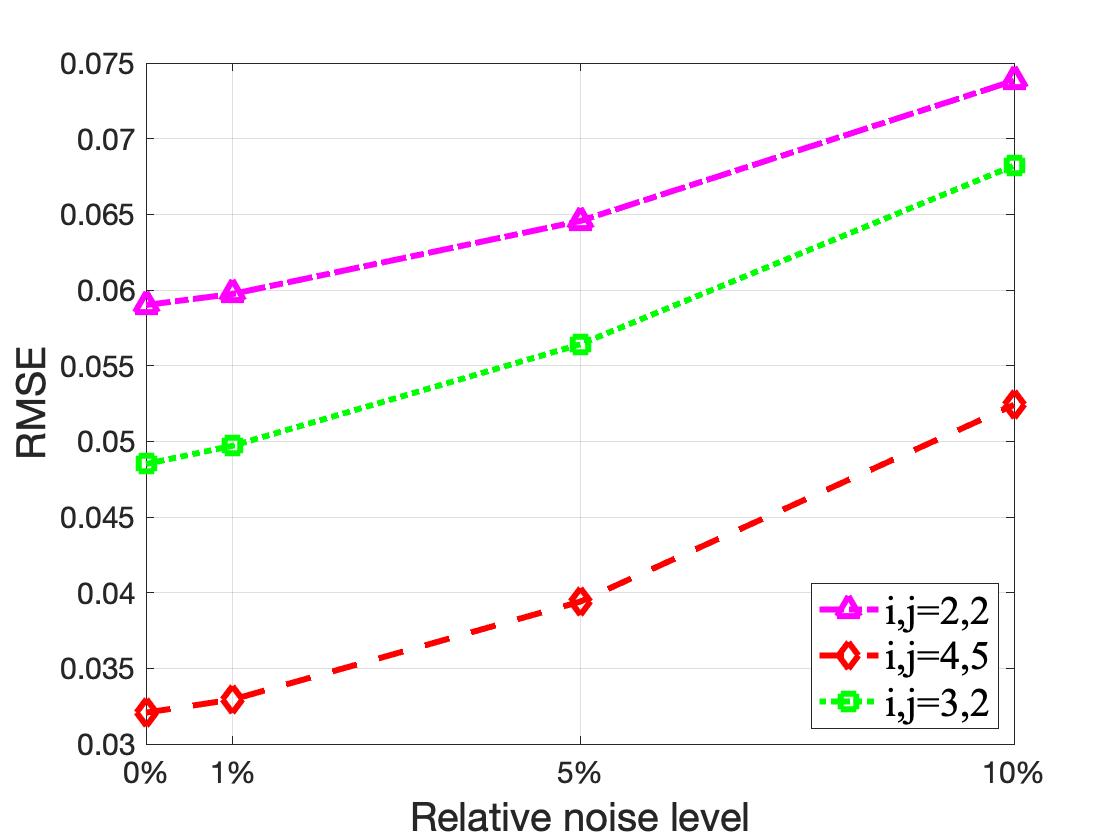}
\caption{Relative RMSE with respect to different noise levels for three tests.}\label{Fig:relative}
\end{figure}
Finally,  by including several numerical experiments, we present in Figure \ref{Fig:relative} the average relative root mean square error (RMSE) in some 2D cases across different noise levels. As the measurement noise escalates, a corresponding increase in the relative error is observed. This trend further corroborates the Lipschitz stability established in our Theorem \ref{thm-BP}, indicating a degree of sensitivity to noise propagation.

%%%%%%%%%%%%%%%%%%%%%%%%%%%%%%%%%%%%%%%%
\appendix
\section{Eigensystem and Fractional Powers of $\cA$}\label{sec-app}

This appendix is devoted to the verification of several important facts about the operator $\cA$ introduced in Section \ref{sec-main}, namely, the existence of an eigensystem as well as its fractional powers $\cA^\ga$ for $\ga\in(0,1)$.\medskip

First we investigate the existence of an eigensystem $\{(\la_n,\bm\vp_n)\}$ of $\cA$ satisfying the properties stated in Section \ref{sec-main}. Following an orthodox strategy, we shall verify that
\begin{enumerate}
\item The bilinear form corresponding with $\cA$ satisfies the assumptions in the Lax-Milgram theorem, which guarantees the existence of $\cA^{-1}$.
\item The operator $\cA^{-1}:(L^2(\Om))^K\longrightarrow(L^2(\Om))^K$ is compact, symmetric and positive definite.
\end{enumerate}
To this end, we discuss the cases of \eqref{eq-IBVP1} and \eqref{eq-IBVP2} separately.\medskip

\noindent{\bf Case of \eqref{eq-IBVP1} } First of all, the bilinear form $B[\,\cdot\,,\,\cdot\,]:(H_0^1(\Om))^K\times(H_0^1(\Om))^K\longrightarrow\BR$ corresponding with $\cA$ is obviously
\[
B[\bm f,\bm g]=\int_\Om\left(\sum_{k=1}^K\bm A_k\nb f_k\cdot\nb g_k+\bm C\bm f\cdot\bm g\right)\rd\bm x
\]
for $\bm f=(f_1,\dots,f_K)^\T,\bm g=(g_1,\dots,g_K)^\T\in(H_0^1(\Om))^K$. Recall that for $\bm f\in(L^2(\Om))^K$, we defined
\[
\|\bm f\|_{L^2(\Om)}=(\bm f,\bm f)^{1/2}=\left(\sum_{k=1}^K\|f_k\|_{L^2(\Om)}^2\right)^{1/2}.
\]
Then for $\bm f\in(H_0^1(\Om))^K$, we have
\[
\|\bm f\|_{H^1(\Om)}=\left(\|\bm f\|_{L^2(\Om)}^2+\|\nb\bm f\|_{L^2(\Om)}^2\right)^{1/2}=\left(\|\bm f\|_{L^2(\Om)}^2+\sum_{k=1}^K\|\nb f_k\|_{L^2(\Om)}^2\right)^{1/2}.
\]
For a matrix $\bm P\in\BR^{d\times d}$, denote its matrix $2$-norm by $\|\bm P\|_2$, that is, $|\bm P\bm\xi|\le\|\bm P\|_2|\bm\xi|$ for any $\bm\xi\in\BR^d$. Then for the matrix-valued functions $\bm A_k\in C^1(\ov\Om;\BR^{d\times d})$ and $\bm C\in L^\infty(\Om;\BR^{d\times d})$, we define
\[
\|\bm A_k\|_{C(\ov\Om;\BR^{d\times d})}:=\max_{\bm x\in\ov\Om}\|\bm A_k(\bm x)\|_2,\quad\|\bm C\|_{L^\infty(\Om;\BR^{K\times K})}:=\mathop{\mathrm{ess}\sup}_{\bm x\in\Om}\|\bm C(\bm x)\|_2,
\]
which are abbreviated as $\|\bm A_k\|_{C(\ov\Om)}$ and $\|\bm C\|_{L^\infty(\Om)}$ respectively for simplicity.

For arbitrary $\bm f,\bm g\in(H_0^1(\Om))^K$, now we utilize the triangle inequality and the Cauchy-Schwarz inequality in $\BR^d$ to estimate $B[\bm f,\bm g]$ from above as
\begin{align*}
B[\bm f,\bm g] & \le\int_\Om\left(|\bm C\bm f\cdot\bm g|+\sum_{k=1}^K|\bm A_k\nb f_k\cdot\nb g_k|\right)\rd\bm x\\
%& \le\int_\Om\left(|\bm C\bm f||\bm g|+\sum_{k=1}^K|\bm A_k\nb f_k||\nb g_k|\right)\rd\bm x\\
& \le\int_\Om\left(\|\bm C(\bm x)\|_2|\bm f||\bm g|+\sum_{k=1}^K\|\bm A_k(\bm x)\|_2|\nb f_k||\nb g_k|\right)\rd\bm x\\
& \le\int_\Om\left(\|\bm C\|_{L^\infty(\Om)}|\bm f||\bm g|+\sum_{k=1}^K\|\bm A_k\|_{C(\ov\Om)}|\nb f_k||\nb g_k|\right)\rd\bm x.
\end{align*}
Setting $M:=\max\{\|\bm C\|_{L^\infty(\Om)},\|\bm A_1\|_{C(\ov\Om)},\dots,\|\bm A_K\|_{C(\ov\Om)}\}$, again we employ the Cauchy-Schwarz inequalities in $\BR^{K+1}$ and $L^2(\Om)$ repeatedly to deduce
\begin{align*}
B[\bm f,\bm g] & \le M\int_\Om\left(|\bm f||\bm g|+\sum_{k=1}^K|\nb f_k||\nb g_k|\right)\rd\bm x\\
& \le M\int_\Om\left(|\bm f|^2+\sum_{k=1}^K|\nb f_k|^2\right)^{1/2}\left(|\bm g|^2+\sum_{k=1}^K|\nb g_k|^2\right)^{1/2}\,\rd\bm x\\
%& \le M\left(\|\bm f\|_{L^2(\Om)}^2+\sum_{k=1}^K\|\nb f_k\|_{L^2(\Om)}^2\right)^{1/2}\left(\|\bm g\|_{L^2(\Om)}^2+\sum_{k=1}^K\|\nb g_k\|_{L^2(\Om)}^2\right)^{1/2}\\
& \le M\|\bm f\|_{H^1(\Om)}\|\bm g\|_{H^1(\Om)}.
\end{align*}
On the other hand, according to the Poincar\'e inequality, there exists a constant $C_\Om>0$ depending only on $\Om$ such that
\[
\|\nb h\|_{L^2(\Om)}^2\ge C_\Om\|h\|_{H^1(\Om)}^2,\quad\forall\,h\in H_0^1(\Om).
\]
Then for arbitrary $\bm f\in(H_0^1(\Om))^K$, the definiteness assumptions \eqref{eq-cond-Ak}--\eqref{eq-cond-C} on $\bm A_k$ and $\bm C$ imply
\begin{align}
B[\bm f,\bm f] & =\int_\Om\left(\sum_{k=1}^K\bm A_k\nb f_k\cdot\nb f_k+\bm C\bm f\cdot\bm f\right)\rd\bm x\ge\ka\int_\Om\sum_{k=1}^K|\nb f_k|^2\,\rd\bm x\nonumber\\
& =\ka\sum_{k=1}^K\|\nb f_k\|_{L^2(\Om)}^2\ge\ka\,C_\Om\sum_{k=1}^K\|f_k\|_{H^1(\Om)}^2=\ka\,C_\Om\|\bm f\|_{H^1(\Om)}^2.\label{eq-Best}
\end{align}
Then $B[\,\cdot\,,\,\cdot\,]$ satisfies the assumptions in the Lax-Milgram theorem. Therefore, for any $\bm f\in(L^2(\Om))^K$, there exists a unique $\bm u\in(H_0^1(\Om))^K$ such that
\begin{equation}\label{eq-LM}
B[\bm u,\bm v]=(\bm f,\bm v),\quad\forall\,\bm v\in(H_0^1(\Om))^K,
\end{equation}
which indicates the existence of $\cA^{-1}:(L^2(\Om))^K\longrightarrow(H_0^1(\Om))^K$ such that $\bm u=\cA^{-1}\bm f$.

Second, we show the compactness of $\cA^{-1}$ as an operator from $(L^2(\Om))^K$ to $(L^2(\Om))^K$. Taking $\bm v=\bm u=\cA^{-1}\bm f$ in \eqref{eq-LM}, we use the lower estimate \eqref{eq-Best} to see
\[
\ka\,C_\Om\|\bm u\|_{H^1(\Om)}^2\le B[\bm u,\bm u]=(\bm f,\bm u)\le\|\bm f\|_{L^2(\Om)}\|\bm u\|_{L^2(\Om)}\le\|\bm f\|_{L^2(\Om)}\|\bm u\|_{H^1(\Om)}
\]
and hence
\[
\|\cA^{-1}\bm f\|_{H^1(\Om)}=\|\bm u\|_{H^1(\Om)}\le\f{\|\bm f\|_{L^2(\Om)}}{\ka\,C_\Om}.
\]
Then $\cA^{-1}$ is bounded as an operator from $(L^2(\Om))^K$ to $(H_0^1(\Om))^K$. On the other hand, it follows from the Rellich-Kondrachov compactness theorem that the embedding $H_0^1(\Om)\subset\subset L^2(\Om)$ is compact. Then the identity operator $I:(H_0^1(\Om))^K\longrightarrow(L^2(\Om))^K$ is also compact. Since the composite of a compact operator and a bounded operator is compact, we arrive at the compactness of $\cA^{-1}:(L^2(\Om))^K\longrightarrow(L^2(\Om))^K$.
%as the composite of $\cA^{-1}:(L^2(\Om))^K\longrightarrow(H_0^1(\Om))^K$ and $I:(H_0^1(\Om))^K\longrightarrow(L^2(\Om))^K$.

Third, we examine the symmetry and positive definiteness of $\cA^{-1}$. Due to the symmetry of $\bm A_k$ and $\bm C$, the symmetry $B[\bm u,\bm v]=B[\bm v,\bm u]$ is obvious. Then for any $\bm f,\bm g\in(L^2(\Om))^K$, setting $\bm u=\cA^{-1}\bm f$ and $\bm v=\cA^{-1}\bm g$ in \eqref{eq-LM} yields
\[
(\cA^{-1}\bm f,\bm g)=(\bm u,\bm g)=(\bm g,\bm u)=B[\bm v,\bm u]=B[\bm u,\bm v]=(\bm f,\bm v)=(\bm f,\cA^{-1}\bm g)
\]
or the symmetry of $\cA^{-1}$. Meanwhile, the positivity
\[
(\cA^{-1}\bm f,\bm f)=B[\cA^{-1}\bm f,\cA^{-1}\bm f]\ge0,\quad\forall\,\bm f\in(L^2(\Om))^K
\]
also follows immediately from \eqref{eq-Best}.

Consequently, the general eigenvalue theory for compact operators asserts that all eigenvalues of $\cA^{-1}:(L^2(\Om))^K\longrightarrow(L^2(\Om))^K$ are positive and the corresponding eigenfunctions make up a complete orthonormal system of $(L^2(\Om))^K$.
%Observing that for $\eta\ne0$, we have $\cA^{-1}\bm\vp=\eta\bm\vp$ if and only if $\cA\bm\vp=\la\bm\vp$ for $\la=\eta^{-1}$,
Then we obtained all the desired properties of the eigensystem claimed  in Section \ref{sec-main} in the case of \eqref{eq-IBVP1}.\medskip

\noindent{\bf Case of \eqref{eq-IBVP2} } Similarly, the bilinear form $B'[\,\cdot\,,\,\cdot\,]:(H_0^1(\Om))^K\times(H_0^1(\Om))^K\longrightarrow\BR$ corresponding with $\cA$ in this case becomes
\[
B'[\bm f,\bm g]=\int_\Om\left(\sum_{i,j,k,\ell=1}^d a_{ijk\ell}(\pa_j f_i)\pa_\ell g_k+\bm C\bm f\cdot\bm g\right)\rd\bm x,\quad\bm f,\bm g\in(H_0^1(\Om))^K.
\]
With slight abuse of notation, we denote the Frobenius inner product of two matrices $\bm P=(p_{ij}),\bm Q=(q_{ij})\in\BR^{d\times d}$ as $\bm P\cdot\bm Q:=\sum_{i,j=1}^d p_{ij}q_{ij}$, and the corresponding Frobenius norm as $|\bm P|=(\bm P\cdot\bm P)^{1/2}$. Regarding a fourth order tensor $\BB\in\BR^{d\times d\times d\times d}$ as a linear operator from $\BR^{d\times d}$ to $\BR^{d\times d}$, we denote its operator norm as $\|\BB\|_\op:=\max_{|\bm P|=1}|\BB\bm P|$, so that we have $|\BB\bm P|\le\|\BB\|_\op|\bm P|$ for any $\bm P\in\BR^{d\times d}$. For the tensor-valued function $\BA\in C^1(\ov\Om;\BR^{d\times d\times d\times d})$ in \eqref{eq-IBVP2}, as before we define $\|\BA\|_{C(\ov\Om)}:=\max_{\bm x\in\ov\Om}\|\BA(\bm x)\|_\op$.

Now for $\bm f,\bm g\in(H_0^1(\Om))^K$, we can simplify
\[
B'[\bm f,\bm g]=\int_\Om(\BA\nb\bm f\cdot\nb\bm g+\bm C\bm f\cdot\bm g)\,\rd\bm x
\]
and provide an upper estimate as
\begin{align*}
|B'[\bm f,\bm g]| & \le\int_\Om(|\BA\nb\bm f\cdot\nb\bm g|+|\bm C\bm f\cdot\bm g|)\,\rd\bm x\le\int_\Om(\|\BA(\bm x)\|_\op|\nb\bm f||\nb\bm g|+\|\bm C(\bm x)\|_2|\bm f||\bm g|)\,\rd\bm x\\
& \le\int_\Om\left(\|\BA\|_{C(\ov\Om)}|\nb\bm f||\nb\bm g|+\|\bm C\|_{L^\infty(\Om)}|\bm f||\bm g|\right)\rd\bm x\le M'\int_\Om(|\nb\bm f||\nb\bm g|+|\bm f||\bm g|)\,\rd\bm x\\
& \le M'\int_\Om\left(|\bm f|^2+|\nb\bm f|^2\right)^{1/2}\left(|\bm g|^2+|\nb\bm g|^2\right)^{1/2}\,\rd\bm x
%& \le M'\left(\|\bm f\|_{L^2(\Om)}^2+\|\nb\bm f\|_{L^2(\Om)}^2\right)^{1/2}\left(\|\bm g\|_{L^2(\Om)}^2+\|\nb\bm g\|_{L^2(\Om)}^2\right)^{1/2}\\
\le M'\|\bm f\|_{H^1(\Om)}\|\bm g\|_{H^1(\Om)},
\end{align*}
where we put $M':=\max\{\|\BA\|_{C(\ov\Om)},\|\bm C\|_{L^\infty(\Om)}\}$.

On the other hand, for any $\bm f\in(H_0^1(\Om))^K$, the non-negative definiteness of $\bm C$ implies
\[
B'[\bm f,\bm f]\ge\int_\Om\BA\nb\bm f\cdot\nb\bm f\,\rd\bm x.
\]
Thanks to the full symmetry assumption \eqref{eq-fullsym} of $\BA$, it is not difficult to verify that
\[
\BA\nb\bm f\cdot\nb\bm f=\BA\f{\nb\bm f+(\nb\bm f)^\T}2\cdot\f{\nb\bm f+(\nb\bm f)^\T}2\quad\mbox{a.e.\! in }\Om.
\]
Now that $(\nb\bm f+(\nb\bm f)^\T)/2$ is a symmetric matrix in $\Om$, we can take advantage of the stability condition \eqref{eq-stabcond} of $\BA$ to estimate $B'[\bm f,\bm f]$ from below as
\begin{align}
B'[\bm f,\bm f] & \ge\int_\Om\BA\f{\nb\bm f+(\nb\bm f)^\T}2\cdot\f{\nb\bm f+(\nb\bm f)^\T}2\,\rd\bm x\nonumber\\
& \ge\ka\int_\Om\left|\f{\nb\bm f+(\nb\bm f)^\T}2\right|^2\rd\bm x\ge\ka\,C_\Om'\|\bm f\|_{H^1(\Om)}^2,\label{eq-B'est}
\end{align}
where we employed Korn's inequality (e.g. \cite[Theorem 5.7]{LQ12}) with a constant $C_\Om'>0$ depending only on $\Om$ in the last inequality.

The remaining part of the verification is identically the same as that in the case of \eqref{eq-IBVP1}, which is omitted here.\medskip

Next, we discuss the existence of fractional powers $\cA^\ga$ for $\ga\in(0,1)$ as well as their domain \eqref{eq-domain-fp} and representation \eqref{eq-fracpow} by means of the eigensystem of $\cA$.

Following the standard theory on fractional powers of operators, first we demonstrate that $\cA$ is an m-accretive operator in $(L^2(\Om))^K$ (see e.g. \cite[\S2.1]{T79}). To begin with, we know that $D(\cA)=(H^2(\Om)\cap H_0^1(\Om))^K$ is dense in $(L^2(\Om))^K$. Next, for any $\bm f\in D(\cA)$, it follows immediately from the integration by parts and \eqref{eq-Best}--\eqref{eq-B'est} that
\[
(\cA\bm f,\bm f)=\left.\begin{cases}
B[\bm f,\bm f], & \mbox{Case of \eqref{eq-IBVP1}}\\
B'[\bm f,\bm f], & \mbox{Case of \eqref{eq-IBVP2}}
\end{cases}\right\}\ge0.
\]
Now it remains to check that for any $\la>0$, the range $R(\la I+\cA)=(L^2(\Om))^K$, that is, for any $\bm f\in D(\cA)$ and $\la>0$, there exists $\bm u\in D(\cA)$ such that $(\la I+\cA)\bm u=\bm f$. Employing the eigensystem of $\cA$, we readily see that
\begin{equation}\label{eq-resolvent}
\bm u=(\la I+\cA)^{-1}\bm f=\sum_{n=1}^\infty\f{(\bm f,\bm\vp_n)}{\la+\la_n}\bm\vp_n\in D(\cA)
\end{equation}
is the unique candidate. Then $\cA$ is indeed an m-accretive operator in $(L^2(\Om))^K$.

Now according to \cite[\S2.6, Theorem 6.9]{P83}, $\cA$ admits fractional powers $\cA^\ga$ for any $\ga\in(0,1)$ such that
\begin{equation}\label{eq-def-fp}
\cA^\ga\bm f=\frac{\sin\pi\ga}\pi\int_0^\infty\la^{\ga-1}\cA(\la I+\cA)^{-1}\bm f\,\rd\la,\quad\bm f\in D(\cA).
\end{equation}
We remark that $\cA^\ga:D(\cA)\longrightarrow(L^2(\Om))^K$ is a bounded operator. In fact, it follows from \cite[\S2.6, Theorem 6.10]{P83} that there exists a constant $C_0>0$ such that
\[
\|\cA^\ga\bm f\|_{L^2(\Om)}\le C_0\|\bm f\|_{L^2(\Om)}^{1-\ga}\|\cA\bm f\|_{L^2(\Om)}^\ga,\quad\bm f\in D(\cA).
\]
Since $\cA\bm f$ involves the derivatives of $\bm f$ up to the second order, there exists a constant $C_\cA>0$ depending only on the coefficients of $\cA$ such that $\|\cA\bm f\|_{L^2(\Om)}\le C_\cA\|\bm f\|_{D(\cA)}$. Together with the fact that $\|\bm f\|_{L^2(\Om)}\le\|\bm f\|_{D(\cA)}$, we obtain
\[
\|\cA^\ga\bm f\|_{L^2(\Om)}\le C_0\|\bm f\|_{D(\cA)}\left(C_\cA\|\bm f\|_{D(\cA)}\right)^{1-\ga}=C_0C_\cA^{1-\ga}\|\bm f\|_{D(\cA)}.
\]

In order to show that $\cA^\ga\bm f$ takes the form of \eqref{eq-fracpow}, for any $\bm f\in D(\cA)$ we set
\[
\bm f_N:=\sum_{n=1}^N(\bm f,\bm\vp_n)\bm\vp_n,\quad N=1,2,\dots.
\]
Then $\bm f_N\longrightarrow\bm f$ in $D(\cA)$ as $N\to\infty$. Substituting $\bm f_N$ into \eqref{eq-def-fp} and using the representation \eqref{eq-resolvent}, we calculate
\begin{align}
\cA^\ga\bm f_N & =\f{\sin\pi\ga}\pi\int_0^\infty\la^{\ga-1}\cA\left(\sum_{n=1}^N\f{(\bm f,\bm\vp_n)}{\la+\la_n}\bm\vp_n\right)\rd\la\nonumber\\
& =\f{\sin\pi\ga}\pi\int_0^\infty\la^{\ga-1}\sum_{n=1}^N\f{\la_n(\bm f,\bm\vp_n)}{\la+\la_n}\bm\vp_n\,\rd\la=\f{\sin\pi\ga}\pi\sum_{n=1}^N I_n\la_n(\bm f,\bm\vp_n)\bm\vp_n,\label{eq-comput}
\end{align}
where
\[
I_n:=\int_0^\infty\f{\la^{\ga-1}}{\la+\la_n}\,\rd\la,\quad n=1,\dots,N.
\]
%Since the integrand of $I_n$ satisfies
%\[
%0\le\f{\la^{\ga-1}}{\la+\la_n}\le\begin{cases}
%\la^{\ga-1}/\la_n, & \la\ll1,\\
%\la^{\ga-2}/2, & \la\gg1,
%\end{cases}
%\]
%we see the convergence of each improper integral $I_n$.
We perform two changes of variables $\la=\la_n\eta$ and $\xi=(\eta+1)^{-1}$ sequentially to calculate
\[
I_n
%=\int_0^\infty\f{(\la_n\eta)^{\ga-1}}{\la_n(\eta+1)}\la_n\,\rd\eta
=\la_n^{\ga-1}\int_0^\infty\f{\eta^{\ga-1}}{\eta+1}\,\rd\eta
%\]
%Further introducing  or equivalently $\eta=(1-\xi)/\xi$, we obtain
%\[
%I_n=\la_n^{\ga-1}\int_0^1\xi\left(\f{1-\xi}\xi\right)^{\ga-1}\left(-\f1{\xi^2}\right)\rd\xi
=\la_n^{\ga-1}\int_0^1\xi^{-\ga}(1-\xi)^{\ga-1}\,\rd\xi
%& =\la_n^{\ga-1}B(\ga,1-\ga)
=\la_n^{\ga-1}\Ga(\ga)\Ga(1-\ga)=\f{\la_n^{\ga-1}\pi}{\sin\pi\ga},
\]
where
%$B(\,\cdot\,,\,\cdot\,)$ stands for the Beta function and
we utilized Euler's reflection formula in the last equality. Hence, plugging the above identity into \eqref{eq-comput} yields
\[
\cA^\ga\bm f_N=\f{\sin\pi\ga}\pi\sum_{n=1}^N\f{\la_n^{\ga-1}\pi}{\sin\pi\ga}\la_n(\bm f,\bm\vp_n)\bm\vp_n=\sum_{n=1}^N\la_n^\ga(\bm f,\bm\vp_n)\bm\vp_n.
\]

We claim that $\{\cA^\ga\bm f_N\}$ is a Cauchy sequence in $(L^2(\Om))^K$, so that there exists the limit
\begin{equation}\label{eq-limit}
\lim_{N\to\infty}\cA^\ga\bm f_N=\lim_{N\to\infty}\sum_{n=1}^N\la_n^\ga(\bm f,\bm\vp_n)\bm\vp_n=\sum_{n=1}^\infty\la_n^\ga(\bm f,\bm\vp_n)\bm\vp_n\quad\mbox{in }(L^2(\Om))^K.
\end{equation}
To see this, we take sufficiently large $N,N'\in\BN$ such that $N'>N$ and $\la_n\ge1$ for all $n\ge N$ without loss of generality. Then $\la_n^\ga \leq\la_n$ for all $n\ge N$ and owing to the orthogonality of $\{\bm\vp_n\}$, we have
\begin{align*}
\|\cA^\ga\bm f_{N'}-\cA^\ga\bm f_N\|_{L^2(\Om)} & =\left\|\sum_{n=N+1}^{N'}\la_n^\ga(\bm f,\bm\vp_n)\bm\vp_n\right\|_{L^2(\Om)}=\left(\sum_{n=N+1}^{N'}|\la_n^\ga(\bm f,\bm\vp_n)|^2\right)^{1/2}\\
& \le\left(\sum_{n=N+1}^{N'}|\la_n(\bm f,\bm\vp_n)|^2\right)^{1/2}=\|\cA\bm f_{N'}-\cA\bm f_N\|_{L^2(\Om)}\\
& \le C_\cA\|\bm f_{N'}-\bm f_N\|_{D(\cA)}\longrightarrow0\quad(N,N'\to\infty).
\end{align*}
Combining \eqref{eq-limit} with the boundedness of  $\cA^\ga:D(\cA)\longrightarrow(L^2(\Om))^K$, we arrive at
\[
\cA^\ga\bm f=\cA\left(\lim_{N\to\infty}\bm f_N\right)=\lim_{N\to\infty}\cA^\ga\bm f_N=\sum_{n=1}^\infty\la_n^\ga(\bm f,\bm\vp_n)\bm\vp_n\in(L^2(\Om))^K
\]
for any $\bm f\in D(\cA)$. Finally, since the above equation holds for any $\bm f\in(L^2(\Om))^K$ as long as the series on the right-hand side converges in $(L^2(\Om))^K$, we can extend the domain of $\cA^\ga$ to $D(\cA^\ga)$ defined in \eqref{eq-domain-fp}. These validate the operator structure of $\cA^\ga$ described in Section \ref{sec-main}.\bigskip

%%%%%%%%%%%%%%%%%%%%%%%%%%%%%%%%%%%%%%%%
{\bf Acknowledgements } The first and third authors are supported by Key-Area Research and Development Program of Guangdong Province (No.2021B0101190003), NSFC (No.11925104), and Science and Technology Commission of Shanghai Municipality (23JC1400501). The second author is supported by JSPS KAKENHI Grant Numbers JP22K13954 and JP23KK0049.

%%%%%%%%%%%%%%%%%%%%%%%%%%%%%%%%%%%%%%%%

\end{document}